\documentclass[onefignum,onetabnum]{siamonline220329}

\usepackage{lipsum}
\usepackage{amsfonts}
\usepackage{graphicx}
\usepackage{epstopdf}
\usepackage{algorithmic}
\usepackage{amssymb}
\usepackage{mathtools}
\usepackage{tikz}
\usepackage{bbm}
\ifpdf
  \DeclareGraphicsExtensions{.eps,.pdf,.png,.jpg}
\else
  \DeclareGraphicsExtensions{.eps}
\fi

\usepackage{tikz}
\usepackage{pgfplots}
\usetikzlibrary{shapes.arrows, patterns, calc}
\usepackage{tikz-3dplot}
\ifpdf
  \DeclareGraphicsExtensions{.eps,.pdf,.png,.jpg}
\else
  \DeclareGraphicsExtensions{.eps}
\fi

\usepgfplotslibrary{groupplots}
\usetikzlibrary{arrows}
\usetikzlibrary{shapes}
\usetikzlibrary{decorations.text}
\usetikzlibrary{quantikz}
\usepackage{siunitx}
\usetikzlibrary{arrows.meta}

\pgfplotsset{ compat=1.18,
    standard/.style={
    scale only axis,
    width=0.5\textwidth,
    enlarge x limits=0.05,
    enlarge y limits=0.05,
    max space between ticks=40,
    every axis/.append style={font=\small},
	every legend/.append style={font=\small},
	every node/.append style={font=\small},	
	}
}

\definecolor{steelblue}{HTML}{A1BDC7}
\definecolor{orange}{HTML}{D98C21}
\definecolor{silver}{HTML}{B0ABA8}
\definecolor{rust}{HTML}{B8420F}
\definecolor{seagreen}{HTML}{2E6B69}
\definecolor{joshua}{HTML}{FBDC7F}
\definecolor{darksky}{HTML}{154c79}

\colorlet{lightsilver}{silver!30!white}
\colorlet{darkorange}{orange!85!black}
\colorlet{darksilver}{silver!85!black}
\colorlet{darksteelblue}{steelblue!85!black}
\colorlet{darkrust}{rust!85!black}
\colorlet{darkseagreen}{seagreen!85!black}

\hypersetup{colorlinks=true,linkcolor=darkrust,citecolor=darkseagreen,urlcolor=darksilver}

\usepackage{mathtools}
\usepackage{bbm}
\usepackage{amsmath}
\usepackage{amssymb}
\usepackage{stackengine}
\usepackage{fixmath}
\usepackage{bm}
\mathtoolsset{centercolon}  

\renewcommand{\phi}{\varphi}
\renewcommand{\epsilon}{\varepsilon}

\newcommand{\cnst}[1]{\mathrm{#1}}  
\newcommand{\econst}{\mathrm{e}}

\newcommand{\Id}{\mathbf{I}} 

\providecommand{\mathbbm}{\mathbb} 

\newcommand{\R}{\mathbbm{R}}

\newcommand{\N}{\mathbbm{N}}

\newcommand{\Expect}{\operatorname{\mathbb{E}}}

\newcommand{\vct}[1]{\mathbold{#1}}
\newcommand{\mtx}[1]{\mathbold{#1}}
\newcommand{\tp}{{T}}

\newcommand{\trace}{\operatorname{tr}}
\newcommand{\vspan}{\operatorname{span}}

\DeclarePairedDelimiterX{\infdivx}[2]{(}{)}{%
  #1\;\delimsize\|\;#2%
}

\newcommand{\rS}{\operatorname{\mathcal{S}}}

\newcommand{\Bs}{\operatorname{\mathcal{B}}}
\newcommand{\Func}{F}
\newcommand{\Probs}{\mathcal{P}}
\DeclareRobustCommand{\stirling}{\genfrac\{\}{0pt}{}}

\newcommand{\allconstant}[2]{\mathbb{I}^{\geq}_{#1}\left(#2\right)}
\newcommand{\allrespects}[2]{\mathbb{I}^{=}_{#1}\left(#2\right)}
\newcommand{\vecpart}{\varsigma}

\newcommand\Ccancel[2][black]{
    \let\OldcancelColor\CancelColor
    \renewcommand\CancelColor{\color{#1}}
    \cancel{#2}
    \renewcommand\CancelColor{\OldcancelColor}
}

\newcommand{\bigO}{\mathcal{O}} 
\graphicspath{{./figures/tikz}}

\usepackage{enumitem}
\setlist[enumerate]{leftmargin=.5in}
\setlist[itemize]{leftmargin=.5in}

\newcommand{\univec}{\mathbf{e}}
\DeclareRobustCommand{\stirling}{\genfrac\{\}{0pt}{}}

\newsiamremark{remark}{Remark}
\newsiamremark{hypothesis}{Hypothesis}
\crefname{hypothesis}{Hypothesis}{Hypotheses}
\newsiamthm{claim}{Claim}

\headers{M{\"o}bius inversion and the iterated bootstrap}{F. Sch{\"a}fer}

\title{M{\"o}bius inversion and the iterated bootstrap} 

\author{Florian Sch{\"a}fer\thanks{Courant Institute of Mathematical Sciences, New York University 
  (\email{florian.schaefer@nyu.edu}).}}

\usepackage{amsopn}

\ifpdf
\hypersetup{
  pdftitle={M{\"o}bius inversion and the iterated bootstrap},
  pdfauthor={F. Sch{\"a}fer}
}
\fi

\begin{document}

\maketitle

\begin{abstract}
  Estimating nonlinear functionals of probability distributions from samples is a fundamental statistical problem.
  The ``plug-in'' estimator obtained by applying the target functional to the empirical distribution of samples is biased.
  Resampling methods such as the bootstrap derive artificial datasets from the original one by resampling.
  Comparing the outcome of the plug-in estimator in the original and resampled datasets allows estimating and thus correcting the bias.
  In the asymptotic setting, iterations of this procedure attain an arbitrarily high order of bias correction, but finite sample results are scarce.
  This work develops a new theoretical understanding of bootstrap bias correction by viewing it as an iterative linear solver for the combinatorial operation of M{\"o}bius inversion.
  It sharply characterizes the regime of linear convergence of the bootstrap bias reduction for moment polynomials. 
  It uses these results to show its superalgebraic convergence rate for band-limited functionals.
  Finally, it derives a modified bootstrap iteration enabling the unbiased estimation of unknown order-$m$ moment polynomials in $m$ bootstrap iterations.
\end{abstract}



\section{Introduction and overview}
\subsection{Background: Bootstrap bias correction}
\subsubsection*{Estimation of nonlinear functionals}
Many statistical problems require the estimation of a nonlinear\footnote{Here and in what follows, we consider nonlinearity with respect to the vector space structure of measures.} functional $\Func$ of a probability distribution $P$ from samples $X \sim P$.
Examples include the estimation of variances of statistical estimators, the estimation of entropies of probability distributions, and the inversion of means of random matrices \cite{valiant2011estimating,lopes2019estimating,derezinski2021sparse,palmer2022calibration,ma2022correcting,epperly2024efficient}.

\subsubsection*{The plug-in estimator and its bias} 
Given $N$ independent samples from an $m$-variate distribution $P$, summarized in a data matrix $\mtx{X} \in \R^{N \times m}$, a natural estimator of $\Func(P)$ is the plug-in estimator $\Func(\mtx{X})$ obtained by applying $\Func$ to the empirical distribution of samples.
But for nonlinear $\Func$, it is biased in the sense that $\Expect_{\mtx{X} \sim_{N} P}\left[\Func(\mtx{X})\right] \neq \Func(P)$.
This can be overcome by estimating and subtracting the bias, a procedure known as bias correction \cite{cordeiro1991bias,an2020resampling,dhaene2022resampling}. 

\subsubsection*{Bootstrap bias correction\nopunct} 
creates new datasets $\mtx{X}^{2} \sim_{N} \mtx{X}$ of the same size obtained by resampling the original data \cite{efron1992bootstrap}. 
The bias of the plug-in estimator $\Func(\mtx{X}^{2})$ applied to $\Func(\mtx{X})$ serves as a proxy for that of the plug-in estimator $\Func(\mtx{X})$ for $\Func(P)$, resulting in the corrected plug-in estimator $[\Bs\Func](\mtx{X}) \coloneqq \Func(\mtx{X}) + (\Func(\mtx{X}) - \Expect_{\mtx{X}^2 \sim_N \mtx{X}}[\Func(\mtx{X}^2))]$.
Iterative application of this idea can further reduce the bias, resulting in $\Bs^2 \Func, \Bs^3 \Func, \ldots, \Bs^k \Func$, the so-called $k$-times iterated (nonparametric) bootstrap \cite{efron1983estimating,hall1988bootstrap,hall2013bootstrap}.

\subsubsection*{Asymptotic theory}
Most theoretical studies of the iterated bootstrap derive asymptotic rates at which the bias of the corrected estimator decreases under increasing $N$, for fixed $P$, $\Func$, and $k$. 
Assumptions on the smoothness of $\Func$ and the finiteness of moments of $P$ yield asymptotic rates of the form $\bigO(N^{-(k + 1)})$ for the $k$-times iterated bootstrap \cite{hall1988bootstrap,hall2013bootstrap}.
For convex $\Func$, \cite{ma2022correcting} prove that bootstrap bias correction improves the mean squared error.

\subsubsection*{Finite sample results\nopunct} are scarce and restricted to special settings.
A key observation of \cite{jiao2020bias} is that the iterated bootstrap amounts to a Neumann series approximation of the inverse of the \emph{resampling operator} $\rS: \Func(\cdot) \mapsto \Expect_{\mtx{X} \sim_N (\cdot)}[\Func(\mtx{X})]$ acting on functionals of probability distributions.
Using this perspective, \cite{jiao2020bias} derives finite sample results for smooth functionals of Bernoulli distributions.
\cite{koltchinskii2020asymptotically,koltchinskii2022estimation,koltchinskii2021estimation} extend this approach to obtain finite sample results for debiasing smooth functionals of empirical covariance matrices of normal distributions.

\subsection{This work: Improved finite sample bounds by exploiting the combinatorics of \texorpdfstring{$\rS$}{S}}
\subsubsection*{A basis for \texorpdfstring{$\rS$}{S}} The starting point of the present work is the observation that the resampling operator $\rS$ maps the space of moment polynomials of a given order onto itself. 
Its restriction to this space is thus represented by a matrix $\mtx{S}$. 
Applying the $k$-times iterated bootstrap to a moment polynomial $\Func$ thus amounts to using $k$ steps of Richardson iteration to solve a linear system in $\mtx{S}$.
The eigenvalues of $\mtx{S}$ encode the effectiveness of the resulting bias reduction.

\subsubsection*{M{\"o}bius inversion} Moment products of total order $m$ are in one-to-one correspondence with the set of partitions of $\{1, \ldots, m\}$. 
The partial order induced by sub-partition turns it into a lattice, a partially ordered set with unique suprema and infima. 
In analogy to antiderivatives on the real line, the antiderivative of a function $\vct{f}$ on a lattice at the point $\pi$ is the partial sum over function values in smaller elements $\sigma \leq \pi$.
The inverse of this operation is called M{\"o}bius inversion. 
In the case of the partition lattice, M{\"o}bius inversion maps moment products to cumulant products \cite{mccullagh2018tensor}.

\subsubsection*{M{\"o}bius inversion and the bootstrap}
The key observation of this work is that, up to a diagonal scaling, the action of $\mtx{S}$ on moment polynomials coincides with the integration of functions on the partition lattice. 
Conversely, the iterated bootstrap implements an iterative algorithm for M{\"o}bius inversion.
Thus, $\mtx{S}$ is a lower triangular matrix. 
Its diagonal entries are computable in closed form and equal its eigenvalues, enabling a detailed understanding of bootstrap bias reduction.

\begin{figure}
  \label{fig:plug-in_bias}
  \centering
  \begin{tikzpicture}
\begin{groupplot}[
	compat=1.3,
	group style={group size=3 by 1,
	horizontal sep=1.5cm,
    vertical sep=0.5cm,},
	]
	\nextgroupplot[
		standard,
 		xlabel={},
		ylabel={},
		height=0.225\textwidth,
		width=0.225\textwidth,
        ymin=-0.00,
        ymax=1.15,
        xmin=-0.50,
        xmax=1.50,
        yticklabels={{$\Func(\mtx{X}^0)$}},
        xticklabels={{$\Expect_{\vct{x} \sim \mtx{X}^0}[\vct{x}]$}},
		xtick={0.5},
		ytick={1.0125},
		xtick style={draw=none},
		ytick style={draw=none},
		enlarge x limits=0.,
		enlarge y limits=0.,
		legend style={
			at={(1.85,1.00)},inner sep=3pt,anchor=south,legend columns=5,legend cell align={left}, draw=none,fill=none, /tikz/every even column/.append style={column sep=0.5cm}},
		]	

		\addlegendimage{steelblue, ultra thick, dashed}
		\addlegendimage{rust, ultra thick, dashed}
		\addlegendimage{orange,ultra thick}
		\addlegendimage{area legend, color=darksky, very thick, pattern = north west lines, pattern color=darksky}
		\addlegendimage{area legend, color=seagreen, very thick, pattern = north east lines, pattern color=seagreen}

		\legend{$\Phi$, Expectation $\mathbb{E}$, Functional $\Func = \Phi(\mathbb{E})$, $\mtx{X}^0$, $\mtx{X}^1 \sim_N \mtx{X}^0$}

		\addplot[orange,ultra thick,domain=-1:2,samples=100] {5 * x^2 * (x - 1)^2 - 2 * (x - 0.5)^2 + 0.7};

		\draw[very thick, darksky, fill, pattern=north west lines, pattern color=darksky] (axis cs: -0.10,0.0) rectangle (axis cs: 0.10,0.64);
		\draw[very thick, darksky, fill, pattern=north west lines, pattern color=darksky] (axis cs: 0.90,0.0) rectangle (axis cs: 1.10,0.64);
		\draw[ultra thick, steelblue, dashed] (axis cs: 0.5,0.0) -- (axis cs: 0.5,2.0);
		\draw[ultra thick, rust, dashed] (axis cs: -0.5,1.0125) -- (axis cs: 1.5,1.0125);

	\nextgroupplot[
		standard,
 		xlabel={},
		ylabel={},
		height=0.225\textwidth,
		width=0.225\textwidth,
        ymin=-0.00,
        ymax=1.15,
        xmin=-0.50,
        xmax=1.50,
        yticklabels={{$\Func(\mtx{X}^1)$}},
        xticklabels={{$\Expect_{\vct{x} \sim \mtx{X}^1}[\vct{x}]$}},
		xtick={0.40625},
		ytick={0.973335},
		xtick style={draw=none},
		ytick style={draw=none},
		enlarge x limits=0.,
		enlarge y limits=0.,
		]	

		\addplot[orange,ultra thick,domain=-1:2,samples=100] {5 * x^2 * (x - 1)^2 - 2 * (x - 0.5)^2 + 0.7};

		\draw[very thick, seagreen, fill, pattern=north east lines, pattern color=seagreen] (axis cs: -0.10,0.0) rectangle (axis cs: 0.10,0.76);
		\draw[very thick, seagreen, fill, pattern=north east lines, pattern color=seagreen] (axis cs: 0.90,0.0) rectangle (axis cs: 1.10,0.52);
		\draw[ultra thick, steelblue, dashed] (axis cs: 0.40625,0.0) -- (axis cs: 0.40625,2.0);
		\draw[ultra thick, rust, dashed] (axis cs: -0.5,0.973335) -- (axis cs: 1.5,0.973335);

	\nextgroupplot[
		standard,
 		xlabel={},
		ylabel={},
		height=0.225\textwidth,
		width=0.225\textwidth,
        ymin=-0.00,
        ymax=1.15,
        xmin=-0.50,
        xmax=1.50,
        yticklabels={{$\Func(\mtx{X}^1)$}},
        xticklabels={{$\Expect_{\vct{x} \sim \mtx{X}^1}[\vct{x}]$}},
		xtick={0.8125},
		ytick={0.620731},
		xtick style={draw=none},
		ytick style={draw=none},
		enlarge x limits=0.,
		enlarge y limits=0.,
		]

		\addplot[orange,ultra thick,domain=-1:2,samples=100] {5 * x^2 * (x - 1)^2 - 2 * (x - 0.5)^2 + 0.7};

		\draw[very thick, seagreen, fill, pattern=north east lines, pattern color=seagreen] (axis cs: -0.10,0.0) rectangle (axis cs: 0.10,0.28);
		\draw[very thick, seagreen, fill, pattern=north east lines, pattern color=seagreen] (axis cs: 0.90,0.0) rectangle (axis cs: 1.10,1.04);
		\draw[ultra thick, steelblue, dashed] (axis cs: 0.8125,0.0) -- (axis cs: 0.8125,2.0);
		\draw[ultra thick, rust, dashed] (axis cs: -0.5,0.620731) -- (axis cs: 1.5,0.620731);
\end{groupplot}

\end{tikzpicture}
  \caption{\textbf{Biased plug-in estimators.} Nonlinear functions $\Phi$ of moments of empirical distributions $\mtx{X}^1 \sim_N \mtx{X}^0$ are biased estimators of the same function of moments of the population distribution $\mtx{X}^0$.}
\end{figure}

\subsubsection*{Theoretical and algorithmic improvements} The M{\"o}bius perspective enables a sharp characterization of the bias reduction of each additional bootstrap iteration, as a function of both $N$ and $m$.
Polynomial approximation yields superalgebraic convergence rates of the bootstrap bias reduction of band-limited functionals, improving bias reduction estimates previously obtained by \cite{koltchinskii2020asymptotically,koltchinskii2021estimation,koltchinskii2022estimation}. 
The detailed understanding of the spectrum of $\mtx{S}$ allows replacing Richardson iteration with more efficient iterations. 
In particular, it yields a modified bootstrap iteration enabling the unbiased estimation of order $m$ moment polynomials (with unknown coefficients) in $m$ iterations.

\section{The Bootstrap}
\label{sec:bootstrap}
\subsection*{Notation}
For a positive integer $m$, we denote as $\Probs(\R^m)$ the set of probability measures on $\R^m$ with finite moments of arbitrary order.
For a distribution $P \in \Probs(\R^m)$ and a positive integer $N$ we write $\mtx{X} \sim_{N} P$ to denote an $N \times m$ matrix containing $N$ independent samples from $P$.
Every such $\mtx{X} \in \R^{N \times m}$ represents an element of $\Probs(\R^m)$ given by the empirical distribution $\left(\delta_{X_{1, :}} + \ldots + \delta_{X_{N, :}}\right)/N$ of its rows.
Thus, $\mtx{X}$ is both a random $N \times m$ data matrix and a random element of $\Probs(\R^m)$. 
The random variable $\mtx{Y} \sim_N \mtx{X}$ is then obtained by sampling $N$ rows of $\mtx{X}$, with replacement. 
Notice that two data matrices of equal dimension represent the same element of $\Probs(\R^m)$ if and only if one is a row permutation of the other. 
We abuse notation slightly to denote all elements of $\Probs(\R^m)$ by bold-face capital letters as if they were represented by a data matrix. 
We denote as $\mtx{X}^{0} \in \Probs(\R^m)$ the (unavailable) population distribution and use superscript indices to denote resampling order, i.e. $\mtx{X}^{i + 1} \sim_N \mtx{X}^{i}$.
Throughout, $\Func \in \R^{\Probs(\R^m)}$ denotes a function from $\Probs(\R^m)$ to the real numbers.
We write $\mtx{x} \sim \mtx{X}$ if the random variable $\vct{x} \in \R^m$ is drawn from the probability distribution (or data matrix) $\mtx{X}$ and denote as $x_{i}$ its $i$-th component.

\subsubsection*{The bias of the plugin estimator}
We are concerned with estimating a functional $F$ of an unknown population distribution $\mtx{X}^0$ from $N$ i.i.d. samples of $\mtx{X}^0$, i.e. a draw from $\mtx{X}^1 \sim_N \mtx{X}^0$. 
A natural approach is the so-called plug-in estimator $F(\mtx{X}^1)$. 
But as illustrated in \cref{fig:plug-in_bias}, this estimator is biased, meaning that $\Expect\limits_{\mtx{X}^1 \sim_N \mtx{X}^0}\left[F(\mtx{X}^1)\right] \neq F(\mtx{X}^0)$. 
This situation is generic whenever $F(\mtx{X})$ is a nonlinear function of statistics of $\mtx{X}$ and thus nonlinear with respect to the vector space structure of measures. 
Examples are covariance matrices, parametric estimates of the Shannon entropy, and inverses of random matrices \cite{valiant2011estimating,derezinski2021sparse,palmer2022calibration,ma2022correcting,epperly2024efficient}.

\begin{figure}
  \label{fig:bootstrap_bias}
  \begin{tikzpicture}
\begin{groupplot}[
	compat=1.3,
	group style={group size=3 by 1,
	horizontal sep=1.5cm,
    vertical sep=0.5cm,},
	]
	\nextgroupplot[
		standard,
 		xlabel={},
		ylabel={},
		height=0.225\textwidth,
		width=0.225\textwidth,
        ymin=-0.00,
        ymax=1.15,
        xmin=-0.50,
        xmax=1.50,
        yticklabels={{$\Func(\mtx{X}^0)$}},
        xticklabels={{$\Expect_{\vct{x} \sim \mtx{X}^0}[\vct{x}]$}},
		xtick={0.5},
		ytick={1.0125},
		xtick style={draw=none},
		ytick style={draw=none},
		enlarge x limits=0.,
		enlarge y limits=0.,
		legend style={
			at={(1.80,1.00)},inner sep=3pt,anchor=south,legend columns=6,legend cell align={left}, draw=none,fill=none, /tikz/every even column/.append style={column sep=0.19cm}},
		]	

		\addlegendimage{orange,ultra thick}
		\addlegendimage{steelblue, ultra thick, dashed}
		\addlegendimage{rust, ultra thick, dashed}
		\addlegendimage{area legend, color=darksky, very thick, pattern = north west lines, pattern color=darksky}
		\addlegendimage{area legend, color=seagreen, very thick, pattern = north east lines, pattern color=seagreen}
		\addlegendimage{area legend, color=silver, very thick, pattern = crosshatch, pattern color=silver}

		\legend{$\Phi$, Expectation $\mathbb{E}$, Functional $\Func = \Phi(\mathbb{E})$, $\mtx{X}^0$, $\mtx{X}^1 \sim_N \mtx{X}^0$, $\mtx{X}^2 \sim_N \mtx{X}^1$}

		\addplot[orange,ultra thick,domain=-1:2,samples=100] {5 * x^2 * (x - 1)^2 - 2 * (x - 0.5)^2 + 0.7};

		\draw[very thick, darksky, fill, pattern=north west lines, pattern color=darksky] (axis cs: -0.10,0.0) rectangle (axis cs: 0.10,0.64);
		\draw[very thick, darksky, fill, pattern=north west lines, pattern color=darksky] (axis cs: 0.90,0.0) rectangle (axis cs: 1.10,0.64);
		\draw[ultra thick, steelblue, dashed] (axis cs: 0.5,0.0) -- (axis cs: 0.5,2.0);
		\draw[ultra thick, rust, dashed] (axis cs: -0.5,1.0125) -- (axis cs: 1.5,1.0125);

	\nextgroupplot[
		standard,
 		xlabel={},
		ylabel={},
		height=0.225\textwidth,
		width=0.225\textwidth,
        ymin=-0.00,
        ymax=1.15,
        xmin=-0.50,
        xmax=1.50,
        yticklabels={{$\Func(\mtx{X}^1)$}},
        xticklabels={{$\Expect_{\vct{x} \sim \mtx{X}^1}[\vct{x}]$}},
		xtick={0.40625},
		ytick={0.973335},
		xtick style={draw=none},
		ytick style={draw=none},
		enlarge x limits=0.,
		enlarge y limits=0.,
		]	

		\addplot[orange,ultra thick,domain=-1:2,samples=100] {5 * x^2 * (x - 1)^2 - 2 * (x - 0.5)^2 + 0.7};

		\draw[very thick, seagreen, fill, pattern=north east lines, pattern color=seagreen] (axis cs: -0.10,0.0) rectangle (axis cs: 0.10,0.76);
		\draw[very thick, seagreen, fill, pattern=north east lines, pattern color=seagreen] (axis cs: 0.90,0.0) rectangle (axis cs: 1.10,0.52);
		\draw[ultra thick, steelblue, dashed] (axis cs: 0.40625,0.0) -- (axis cs: 0.40625,2.0);
		\draw[ultra thick, rust, dashed] (axis cs: -0.5,0.973335) -- (axis cs: 1.5,0.973335);

	\nextgroupplot[
		standard,
 		xlabel={},
		ylabel={},
		height=0.225\textwidth,
		width=0.225\textwidth,
        ymin=-0.00,
        ymax=1.15,
        xmin=-0.50,
        xmax=1.50,
        yticklabels={{$\Func(\mtx{X}^2)$}},
        xticklabels={{$\Expect_{\vct{x} \sim \mtx{X}^2}[\vct{x}]$}},
		xtick={0.8125},
		ytick={0.620731},
		xtick style={draw=none},
		ytick style={draw=none},
		enlarge x limits=0.,
		enlarge y limits=0.,		]

		\addplot[orange,ultra thick,domain=-1:2,samples=100] {5 * x^2 * (x - 1)^2 - 2 * (x - 0.5)^2 + 0.7};

		\draw[very thick, silver, fill, pattern=crosshatch, pattern color=silver] (axis cs: -0.10,0.0) rectangle (axis cs: 0.10,0.28);
		\draw[very thick, silver, fill, pattern=crosshatch, pattern color=silver] (axis cs: 0.90,0.0) rectangle (axis cs: 1.10,1.04);
		\draw[ultra thick, steelblue, dashed] (axis cs: 0.8125,0.0) -- (axis cs: 0.8125,2.0);
		\draw[ultra thick, rust, dashed] (axis cs: -0.5,0.620731) -- (axis cs: 1.5,0.620731);
\end{groupplot}

\end{tikzpicture}
  \caption{\textbf{Bootstrap bias correction.} In outcomes where the empirical distribution $\mtx{X}^1$ represents the population $\mtx{X}^0$ well, the bias in the plugin estimate using a resampled $\mtx{X}^2 \sim_{N} \mtx{X}^1$ reveals the bias due to plugin estimation. Subtracting it from the estimates allows accounting for the alternative outcomes where $\mtx{X}^1$ does not represent $\mtx{X}^0$ well. 
  This is the fundamental mechanism underlying bootstrap bias correction.}
\end{figure}

\subsubsection*{Bootstrap bias correction\nopunct} introduces a new random measure $\mtx{X}^2 \sim_N \mtx{X}^1$ obtained from resampling elements of $\mtx{X}^1$ with replacement. 
$\mtx{X}^2$ is obtained from $\mtx{X}^1$ in the same way in which $\mtx{X}^1$ is obtained from $\mtx{X}^0$, suggesting that the bias of the plugin estimator $F(\mtx{X}^2)$ for $F(\mtx{X}^1)$ mimicks that of the plugin estimator $F(\mtx{X}^1)$ for $F(\mtx{X}^0)$ (see \cref{fig:bootstrap_bias}),
\begin{equation}
\Expect\limits_{\mtx{X}^1 \sim_N \mtx{X}^0}\left[F(\mtx{X}^1)\right] - F(\mtx{X}^0) \approx \Expect\limits_{\mtx{X}^1 \sim_N \mtx{X}^0}\left[ \Expect\limits_{\mtx{X}^2 \sim_N \mtx{X}^1} \left[F(\mtx{X}^2)\right] - F(\mtx{X}^1)\right].
\end{equation}
This motivates defining the (affine) bootstrap operator $\Bs: \R^{\Probs(\R^m)} \longrightarrow \R^{\Probs(\R^m)}$ as
\begin{equation}
  [\Bs G](\mtx{X}) \coloneqq (F(\mtx{X}) -  \left(\Expect\limits_{\mtx{Y} \sim_N \mtx{X}} \left[G(\mtx{Y})\right] - G(\mtx{X})\right)
\end{equation}
and using the estimator $[\Bs F](\mtx{X}^1)$ for $F(\mtx{X}^0)$.
The expectation can be approximated by a Monte Carlo sum. 
If bias is the only concern, a single Monte Carlo sample suffices.

\subsubsection*{The iterated bootstrap}
The above argument can be iterated by approximating the bias of the estimator $[\Bs F](\mtx{X}^1)$ for $F(\mtx{X}^0)$ with that of the estimator $[\Bs F](\mtx{X}^2)$ for $F(\mtx{X}^1)$. 
Since $[B\mtx{F}]$ resamples its input, $[B\mtx{F}](\mtx{X}^2)$ involves \emph{two} steps of resampling $\mtx{X}^3 \sim_N \mtx{X}^2 \sim_N \mtx{X}^1$.
Iterating the above argument one obtains the $k$ times iterated bootstrap 
\begin{equation}
  [\Bs^k F](\mtx{X}^1) = \Expect \limits_{\mtx{X}^{k + 1} \sim_N \dots \sim_N \mtx{X}^1} \left[ \sum \limits_{i = 1}^{k + 1} \begin{pmatrix} k + 1\\ i \end{pmatrix} (-1)^{i + 1} F(\mtx{X}^i) \right].
\end{equation}
Under suitable technical conditions and in the limit $N \rightarrow \infty$, the $k$-times iterated bootstrap can reduce the bias to the order $\bigO\left( N^{-(k + 1)}\right)$ \cite{hall2013bootstrap}.  

\subsubsection*{An operator view on the bootstrap} 
As observed by \cite{jiao2020bias,koltchinskii2020asymptotically} the iterated bootstrap approximately inverts the sampling operator $\rS: \R^{\Probs(\R^m)} \longrightarrow \R^{\Probs(\R^m)}$ defined as
\begin{equation}
  [\rS F](\mtx{X}^i) \coloneqq \Expect\limits_{\mtx{X}^{i + 1} \sim_N \mtx{X}^i} \left[F(\mtx{X}^{i + 1})\right].
\end{equation}
\cite{koltchinskii2020asymptotically} also observes that $\rS$ is the adjoint of the transition kernel of the Markov chain on $\Probs(\R^d)$ with transitions given by $\sim_N$.
By definition of $\rS$, the bias of the plugin estimator is given by $F(\mtx{X}^0) - [\rS F]$ and if $\rS^{-1}F$ exists, $[\rS^{-1} F](\mtx{X}^1)$ is an unbiased estimate of $F(\mtx{X}^0)$. 
The iterated bootstrap amounts to a Neumann series approximation of $\rS^{-1},$
\begin{equation}
  \Bs^k F = \sum \limits_{i = 0}^k (\cnst{Id} - \rS)^i F \approx \rS^{-1} F.
\end{equation}
If a Neumann series converges in operator norm, its limit is $\rS^{-1}$, thus proving its existence.
Jiao and Han investigate the properties of $\rS$ and its Neumann series approximation in the setting of quite general $F$ but restricted to the special case of $\mtx{X}^0$ being a binomial distribution \cite{jiao2020bias}. 
Our work complements this approach by treating general $\mtx{X}^0$ in the special case of $F$ being (approximated by) a moment polynomial. 
It relates the bootstrap to the combinatorial procedure of M{\"o}bius inversion, which we now introduce.

\newpage
\section{M{\"o}bius inversion on the partition lattice}

\subsection*{The partition lattice.}
We denote as $\Pi(m)$ the set of all partitions of the set $\{1, \ldots m\}$. 
For $\pi, \sigma \in \Pi(m)$, we write $\pi \leq \sigma$ if $\pi$ is a sub-partition of $\sigma$, obtaining a partial order on $\Pi(m)$.
$\Pi(m)$ with partial order $\leq$ is known as the \emph{partition lattice}.

\subsection*{Moments and cumulants}
To a $p \subset \{1, \ldots, m\}$, we assign $\mu_{p}(\mtx{X}) \coloneqq \Expect \limits_{\vct{x} \sim \mtx{X}}\big[\prod_{i \in p} x_i \big]$, a \emph{moment} of a probability distribution $\mtx{X}$.   
Equivalently, it is defined as the $\prod_{i \in p} \partial_i$ derivative of the moment generating function $\Expect \limits_{\vct{x} \sim \mtx{X}}\left[\exp(\langle x, \cdot \rangle)\right]$ of $\vct{x}$ at zero.  
We further assign to each $p \subset \{1, \ldots, m\}$ a \emph{cumulant} $\kappa_{p}$ that is the  $\prod_{i \in p} \partial_i$ derivative of the log-moment generating function at zero. 
Compared to moments, cumulants have the advantage that they behave predictably under the addition of independent random variables and encode independence and Gaussianity \cite{mccullagh2018tensor}.
Moments and cumulants are polynomials in each other, such as $\kappa_{\{i\}} = \mu_{\{i\}}$, $\kappa_{\{i, j\}} = \mu_{\{i, j\}} - \mu_{\{i\}} \mu_{\{j\}}$, $\kappa_{\{i, j, k\}} = \mu_{\{i, j, k\}} - \mu_{\{i\}} \mu_{\{j, k\}}-  \mu_{\{i, j\}} \mu_{\{k\}} - \mu_{\{i, k\}} \mu_{\{j\}} + 2\mu_{\{i, j, k\}}$.

\begin{figure}
  \label{fig:moebius}
  \centering
  \begin{tikzpicture}
    \begin{scope}[yshift=-4.0cm]
    \node {$\begin{matrix} 
                1| 2 | 3 \\ 
                1  2 | 3 \\ 
                1  3 | 2 \\
                2  3 | 1 \\ 
                1  2   3 
            \end{matrix}$};        
    \end{scope}
    \begin{scope}[yshift=-4.0cm, xshift=3.5cm]
    \node {$\overline{\mtx{S}} = \begin{pmatrix} 
                1 & 0 & 0 & 0 & 0\\
                \color{joshua} 1 & 1 & 0 & 0 & 0\\
                \color{joshua} 1 & 0 & 1 & 0 & 0\\
                \color{joshua} 1 & 0 & 0 & 1 & 0\\
                \color{seagreen} 1 & \color{steelblue} 1 & \color{steelblue}1 & \color{steelblue}1 & 1
            \end{pmatrix}$};        
    \end{scope}

    \begin{scope}[yshift=-4.0cm, xshift=9.0cm]
    \node {$\overline{\mtx{S}}^{-1} = \begin{pmatrix} 
                1 & 0 & 0 & 0 & 0\\
                \color{joshua} - 1 & 1 & 0 & 0 & 0\\
                \color{joshua}- 1 & 0 & 1 & 0 & 0\\
                \color{joshua}- 1 & 0 & 0 & 1 & 0\\
                \color{seagreen}2 & \color{steelblue}-1 & \color{steelblue}- 1 & \color{steelblue}- 1 & 1
            \end{pmatrix}$};        
    \end{scope}
           
    \begin{scope}[xshift=-1cm]
        \node (1 2) at(1, 0.5) {$1|2$};
        \node (12) at(1, -0.5) {$12$};

        \draw [thick] (1 2) -- (12);
    \end{scope}

    \begin{scope}[xshift=1.5cm]
        \node (1 2 3) at(1, 1) {$1|2|3$};
        \node (12 3)  at(0, 0) {$12|3$};
        \node (13 2)  at(1, 0) {$13|2$};
        \node (23 1)  at(2, 0) {$23|1$};
        \node (123)   at(1, -1) {$123$};

        \draw [ultra thick, color=joshua] (1 2 3.south) -- (12 3.north);
        \draw [ultra thick, color=joshua] (1 2 3.south) -- (13 2.north);
        \draw [ultra thick, color=joshua] (1 2 3.south) -- (23 1.north);

        \draw [ultra thick, color=steelblue]  (12 3.south) -- (123.north);
        \draw [ultra thick, color=steelblue]  (13 2.south) -- (123.north);
        \draw [ultra thick, color=steelblue]  (23 1.south) -- (123.north);

    \end{scope}

    \begin{scope}[xshift=8.25cm]
        \node (1 2 3 4) at(0, 2.25) {$1|2|3|4$};
        \node (13 2 4)  at(-2.75, 0.75) {$13|2|4$};
        \node (24 1 3)  at(-1.65, 0.75) {$24|1|3$};
        \node (12 3 4)  at(-0.55, 0.75) {$12|3|4$};
        \node (34 1 2)  at( 0.55, 0.75) {$34|1|2$};
        \node (14 2 3)  at( 1.65, 0.75) {$14|2|3$};
        \node (23 1 4)  at( 2.75, 0.75) {$23|1|4$};

        \node (123 4)  at(-3.3, -0.75) {$123|4$};
        \node (13 24)  at(-2.2, -0.75) {$13|24$};
        \node (124 3)  at(-1.1, -0.75) {$124|3$};
        \node (12 34)  at( 0.0, -0.75) {$12|34$};
        \node (134 2)  at( 1.1, -0.75) {$134|2$};
        \node (14 23)  at( 2.2, -0.75) {$14|23$};
        \node (234 1)  at( 3.3, -0.75) {$234|1$};

        \node (1234) at(0, -2.25) {$1234$};

        \draw [thick] (1 2 3 4.south) -- (13 2 4.north);
        \draw [thick] (1 2 3 4.south) -- (24 1 3.north);
        \draw [thick] (1 2 3 4.south) -- (12 3 4.north);
        \draw [thick] (1 2 3 4.south) -- (34 1 2.north);
        \draw [thick] (1 2 3 4.south) -- (14 2 3.north);
        \draw [thick] (1 2 3 4.south) -- (23 1 4.north);

        \draw [thick] (13 2 4.south) -- (123 4.north);
        \draw [thick] (13 2 4.south) -- (13 24.north);
        \draw [thick] (13 2 4.south) -- (134 2.north);
        \draw [thick] (24 1 3.south) -- (13 24.north);
        \draw [thick] (24 1 3.south) -- (124 3.north);
        \draw [thick] (24 1 3.south) -- (234 1.north);
        \draw [thick] (12 3 4.south) -- (123 4.north);
        \draw [thick] (12 3 4.south) -- (124 3.north);
        \draw [thick] (12 3 4.south) -- (12 34.north);
        \draw [thick] (34 1 2.south) -- (12 34.north);
        \draw [thick] (34 1 2.south) -- (134 2.north);
        \draw [thick] (34 1 2.south) -- (234 1.north);
        \draw [thick] (14 2 3.south) -- (124 3.north);
        \draw [thick] (14 2 3.south) -- (134 2.north);
        \draw [thick] (14 2 3.south) -- (14 23.north);
        \draw [thick] (23 1 4.south) -- (123 4.north);
        \draw [thick] (23 1 4.south) -- (14 23.north);
        \draw [thick] (23 1 4.south) -- (234 1.north);

        \draw [thick] (123 4.south) -- (1234.north);
        \draw [thick] (13 24.south) -- (1234.north);
        \draw [thick] (124 3.south) -- (1234.north);
        \draw [thick] (12 34.south) -- (1234.north);
        \draw [thick] (134 2.south) -- (1234.north);
        \draw [thick] (14 23.south) -- (1234.north);
        \draw [thick] (234 1.south) -- (1234.north);

    \end{scope}
\end{tikzpicture}
  \caption{\textbf{M{\"o}bius inversion on the partition lattice.} The matrix $\overline{\mtx{S}}$ acts on lattice functions, implementing the discrete antiderivative. Its inverse, the M{\"o}bius matrix, implements a lattice version of differentiation. 
  In this work we draw Hasse diagrams of the partition lattice with the finest partition at the top and the coarsest partition at the bottom, which is the opposite of the usual convention.
  This is to match the order of rows and columns of the associated matrices, which is given by the standard partial order of the partition lattice.
  }
\end{figure}

\subsection*{M{\"o}bius inversion\nopunct} allows concisely expressing moment-cumulant relationships. 
For a partially ordered set $\Sigma$ and function $\vct{f} \in \R^{\Sigma}$, we define an indefinite integral $\vct{g} \in \R^{\Sigma}$ by 
\begin{equation}
  g_{\pi} := \sum \limits_{\sigma \leq \pi} f_{\sigma}.  
\end{equation} 
For $\Sigma = \mathbb{Z}$ this is the usual (discrete) antiderivative. 
Defining the matrix $\overline{\mtx{S}} \in \R^{\Sigma \times \Sigma}$ by
\begin{equation}
  \label{eq:mobius_matrix}
  \overline{S}_{\pi \sigma} \coloneqq 
  \begin{cases}
    1, \quad \text{for} \quad \pi \geq \sigma \\
    0, \quad \text{else}
  \end{cases},
\end{equation} 
we can write $\vct{g} = \overline{\mtx{S}} \vct{f}$.  
When ordering the rows and columns of $\overline{\mtx{S}}$ according to $\leq$, it is a lower triangular matrix with unit diagonal and thus invertible. 
Thus, there exists an inverse $\overline{\mtx{S}}^{-1}$ mapping $\vct{g}$ to $\vct{f}$, as illustrated in \cref{fig:moebius}.
For $\Sigma = \mathbb{Z}$, this is the discrete derivative. 
Thus, the \emph{M{\"o}bius function/matrix} $\overline{\mtx{S}}^{-1}$ generalizes differentiation to arbitrary partially ordered sets.

\subsection*{Moments, Cumulants, and M{\"o}bius inversion}
We now consider M{\"o}bius inversion on the partition lattice $\Pi(m)$. 
Define the moment and cumulant products $\vct{\mu}, \vct{\kappa} \in \R^{\Pi(m)}$ by
\begin{equation}
\mu_{\pi} \coloneqq \prod \limits_{s \in \pi} \mu_{s}, \quad \kappa_{\pi} \coloneqq \prod \limits_{s \in \pi} \kappa_{s}.
\end{equation}
Here, in a slight abuse of notation, $\mu,\kappa$ denote moments/cumulants or moment/cumulant products depending on whether the subscript is a subset or partition of $\{1, \ldots, m\}$.
The relationship between moments and cumulants is given by $\vct{\mu} = \overline{\mtx{S}} \vct{\kappa}$, where $\overline{\mtx{S}}^{-1}$ is the M{\"o}bius matrix on $\Pi(m)$.
In other words, the cumulant products are the ``derivatives'' of the moment products. 
The intuition is that the cumulant $\kappa_{\{1, \ldots, m\}}$ expresses the part of $\mu_{\{1, \ldots, m\}}$ that is not already captured by products of lower-order moments.  
\cref{fig:moebius} illustrates the relationship of the M{\"o}bius matrix and its inverse to the partition lattice.

\section{The bootstrap on moment polynomials}
\subsection*{Moment polynomials}
We now study the action of the resampling operator $\rS$ on the finite-dimensional subspace of $\R^{\Probs(\R^m)}$ consisting of \emph{moment polynomials} of the form
\begin{equation}
  F\left(\mtx{X}\right) = \sum_{\pi \in \Pi(m)} f_{\pi} \mu_{\pi} \left(\mtx{X}\right),
\end{equation}
for coefficient vectors $\vct{f} \in \R^{\Pi(m)}$. 
Throughout this section, we assume that $F$ is of this form. 

\begin{figure}
  \label{fig:equivalent_constant}
  \centering
  \newcommand{\elm}[3]{\!\!\!\!\!\color{#1}\frac{\frac{N!}{(N - #2)!}}{N^#3}\!\!\!\!\!}
\begin{tikzpicture}

    \begin{scope}[yshift=-0.0cm, xshift=1.5cm]
    \node {$\vct{i} = 
        \begin{cases}
            \left({\color{seagreen} 3, 3}, {\color{rust} 5}\right) \\
            \left({\color{orange} 1, 1, 1}\right) \\
            \left({\color{steelblue} 2}, {\color{silver} 4, 4}\right)
        \end{cases}
        \ 
        \vecpart\left(\vct{i}\right) = 
        \begin{cases}
            \left\{{\color{seagreen} \{1, 2\}}, {\color{rust} \{3\}}\right\} \\
            \left\{{\color{orange} \{1, 2, 3\}}\right\} \\
            \left\{{\color{steelblue} \{1\}}, {\color{silver} \{2, 3\}}\right\}
        \end{cases}
        \
        \begin{cases}
            \vct{i} \in \allrespects{5}{\left\{{\color{seagreen}\{1, 2\}}, {\color{rust}\{3\}}\right\}}, \ \vct{i} \in \allconstant{5}{\left\{{\color{seagreen}\{1, 2\}}, {\color{rust}\{3\}}\right\}}  \\
            \vct{i} \notin \allrespects{5}{\left\{{\color{seagreen}\{1, 2\}}, {\color{rust}\{3\}}\right\}}, \ \vct{i} \in \allconstant{5}{\left\{{\color{seagreen}\{1, 2\}}, {\color{rust}\{3\}}\right\}}  \\
            \vct{i} \notin \allrespects{5}{\left\{{\color{seagreen}\{1, 2\}}, {\color{rust}\{3\}}\right\}}, \ \vct{i} \notin \allconstant{5}{\left\{{\color{seagreen}\{1, 2\}}, {\color{rust}\{3\}}\right\}}        \end{cases}
        $};        
    \end{scope}
\end{tikzpicture}
  \vspace{-7mm}
  \caption{\textbf{Multiindices and partitions.} To each multiindex $\vct{i} \in \{1, \ldots, N\}^m$, such as the the three examples $(3, 3, 5), (1, 1, 1,), (2, 4, 4) \in \{1, \ldots, 5\}^3$, we can associate a partition $\vecpart\left(\vct{i}\right) \in \Pi(m)$ obtained by assigning to the same part those indices that are mapped to the same value.
  For $\pi \in \Pi(m)$, we write $\vct{i} \in \allrespects{N}{\pi}$ if the multiindex is equivalent to $\pi$ in the sense that $\vecpart(\vct{i}) = \pi$. 
  We write $\vct{i} \in \allconstant{N}{\pi}$ if the multiindex is coarser than $\pi$ in the sense that $\vecpart(\vct{i}) \geq \pi$, meaning that $\vecpart(\vct{i})$ can be obtained by merging parts of $\pi$.}
\end{figure}

\subsection*{Notation on index sets} To every multiindex $\mtx{i} \in \{1, \dots, N\}^m$ we can associate a unique partition $\vecpart(\mtx{i}) \in \Pi(m)$ that assigns indices $k, l \in \{1, \ldots, m\}$ to the same part if and only if $i_k = i_l$.
We denote as $\allconstant{N}{\pi} \coloneqq \left(\mtx{i} \in \{1, \dots, N\}^m \middle| \vecpart(\mtx{i}) \geq \pi\right)$ the set of multiindices coarser than $\pi$, in the sense that $\vecpart\left(\vct{i}\right)$ can be obtained by merging parts of $\pi$. 
We further denote as $\allrespects{N}{\pi} \coloneqq \left(\mtx{i} \in \{1, \dots, N\}^m \middle| \vecpart(\mtx{i}) = \pi\right)$ the set of multiindices equivalent to $\pi$, meaning that $\vecpart\left(\vct{i}\right)$ equals $\pi$. 
For instance, $\mtx{i} = (3, 3, 5)$ is equivalent to $\pi = \{\{1, 2\}, \{3\}\}$, but neither $(1, 1, 1)$ nor $(2, 4, 4)$ are, even though the former is coarser than $\pi$ (see \cref{fig:equivalent_constant} for an illustration).
For $\# \mathbb{I}$ the cardinality of a set $\mathbb{I}$, we see that $\# \allconstant{N}{\pi} = N^{\# \pi}$ and $\# \allrespects{N}{\pi} = \frac{N!}{(N - \# \pi)!}$.
In particular, for $N < \# \pi$, $\allrespects{N}{\pi} = \emptyset$.
The sets $\allconstant{N}{\pi}$ and $\allrespects{N}{\pi}$ are related as follows.

\begin{lemma}
  \label{lem:constantrespect}
  For any $m, N \in \N_{>0}$ and $\pi \in \Pi(m)$, we have  
  \begin{equation}
    \allconstant{N}{\pi} = \mathop{\dot{\bigcup}}_{\pi \leq \tilde{\pi}} \allrespects{N}{\tilde{\pi}}, \quad \quad \quad  \allconstant{N}{\min(\Pi(m))} = \{1, \ldots, N\}^m = \mathop{\dot{\bigcup}}_{\pi \in \Pi(m)} \allrespects{N}{\pi},
  \end{equation}
  where the unions are disjoint and $\min(\Pi(m)) = \{\{i\}\}_{1 \leq i \leq m}$ is finest partition possible.
\end{lemma}
\begin{proof}
  For $\tilde{\pi}, \tilde{\sigma} \in \Pi(m)$ with $\tilde{\pi} \neq \tilde{\sigma}$, we have $\allrespects{N}{\tilde{\pi}} \cap \allrespects{N}{\tilde{\sigma}} = \emptyset$, since for a $\vct{i}$ in that intersection we would have $\vecpart(\vct{i}) = \tilde{\pi}$ \emph{and} $\vecpart(\vct{i}) = \tilde{\sigma}$.
  The result now follows since any $\vct{i} \in \allconstant{N}{\pi}$ satisfies $\vecpart(\vct{i}) \geq \pi$ and $\vct{i} \in \allconstant{N}{\vecpart\left(\vct{i}\right)}$.
\end{proof}

\subsection*{Unbiased estimators of moment products} 
Multiindices equivalent to a partition define unbiased estimators of moment products, the so-called \emph{symmetric statistics} \cite{mccullagh2018tensor}.
\begin{lemma}
  \label{lem:symstat}
  For all $\pi \in \Pi(m)$, $N \geq m$, and $\vct{i} \in \allrespects{N}{\pi}$, we have for $\mtx{Y} \sim_{N} \mtx{X}$,
  \begin{equation}
    \label{eq:symmetric_statistics}
    \mu_{\pi}(\mtx{X}) = \Expect\left[\prod \limits_{j \in \{1, \ldots, m\}} Y_{\vct{i}_j j}\right] \implies     
    \mu_{\pi}(\mtx{X}) = \Expect \left[\frac{1}{\# \allrespects{N}{\pi}}\sum_{\mtx{i} \in \allrespects{N}{\pi}} \prod \limits_{j \in \{1, \ldots, m\}} Y_{\vct{i}_j j}\right].
  \end{equation}
\end{lemma}
\begin{proof}
  The first equality follows by observing that the rows of $\mtx{Y}$ are distributed according to $\mtx{X}$ and independent of each other.
  Thus, we can compute
  \begin{equation*}
    \Expect\left[\prod \limits_{j \in \{1, \ldots, m\}} Y_{\vct{i}_j j}\right]
    = \Expect\left[\prod \limits_{p \in \pi} \left(\prod \limits_{j \in p} Y_{\vct{i}_j j}\right)\right] 
    = \prod \limits_{p \in \pi} \Expect\left[\left(\prod \limits_{j \in p} Y_{\vct{i}_j j}\right)\right] 
    = \prod \limits_{p \in \pi} \mu_p\left(\mtx{X}\right)
    = \mu_{\pi}\left(\mtx{X}\right). 
  \end{equation*}
  The second equality follows from the first by pulling the expectation through the sum.
\end{proof}

\subsection*{The plug-in estimator and its mean}
Our motivation for using bootstrap bias correction is that the plug-in estimator applied to moment polynomials is biased, in the sense that $\Expect_{\mtx{Y} \sim_N \mtx{X}}\left[\mu_{\pi}\left(\mtx{Y}\right)\right]$ differs from $\Expect\left[\mu_\pi \left(\mtx{X}\right)\right]$.
We will now derive an explicit form of the former.
\begin{lemma}
  \label{lem:plug-in}
  For all $\pi \in \Pi(m)$ and $N \geq m$, we have for $\mtx{Y} \sim_{N} \mtx{X}$, 
  \begin{equation}
    \Expect\left[\mu_{\pi}(\mtx{Y})\right] = \Expect \left[\frac{1}{\# \allconstant{N}{\pi}}\sum_{\mtx{i} \in \allconstant{N}{\pi}} \prod \limits_{j \in \{1, \ldots, m\}} Y_{i_j j}\right].
  \end{equation}
\end{lemma}
\begin{proof}
  We begin by computing 
  \begin{align*}
    \Expect\left[\mu_{\pi}(\mtx{Y})\right] 
    = \Expect\left[\prod \limits_{p \in \pi} \left(\Expect_{\vct{y} \sim \mtx{Y}} \left[\prod_{j \in p} y_{j}\right]\right)\right]
    = \Expect\left[\prod \limits_{p \in \pi}\left( \frac{1}{N} \sum_{i = 1}^N \left[\prod_{j \in p} Y_{i, j}\right]\right)\right].
  \end{align*}
  The term inside the expectation is a product of sums. 
  Writing $p_1, \ldots, p_{\#\pi}$ for the parts of $\pi$, we can expand the product of sums to a sum of products as
  \begin{equation}
    \prod \limits_{p \in \pi}\left( \frac{1}{N} \sum_{i = 1}^N \left[\prod_{j \in p} Y_{i, j}\right]\right)
    = \frac{1}{N^{\#\pi}} \sum \limits_{1 \leq i_1, i_2, \ldots, i_{\# \pi} \leq N} \prod \limits_{k = 1}^{\# \pi} \left( \prod \limits_{j \in p_k} Y_{i_k, j}\right)
    = \frac{1}{N^{\#\pi}} \sum \limits_{\mtx{i} \in \allconstant{N}{\pi}} \prod \limits_{j = 1}^m Y_{i_j, j}.
  \end{equation}
  Observing that $\# \allconstant{N}{\pi} = N^{\# \pi}$, we obtain the result. 
\end{proof}
Thus, the bias of the plug-in estimate for the moment product $\mu_{\pi}$ is characterized by replacing $\allrespects{N}{\pi}$ with $\allconstant{N}{\pi}$ in \cref{eq:symmetric_statistics}.
This will now help us understand the resampling operator $\rS$.


\subsection*{The action of $\rS$ on moment products} \cref{lem:plug-in} shows that the expectation of the plug-in estimator is given as a sum of terms of the form $\Expect\left[\prod_{j \in \{1, \ldots m\}} Y_{i_j, j}\right]$, over elements of $\allconstant{N}{\pi}$.
But each such element $\vct{i} \in \allconstant{N}{\pi}$ is itself in $\allrespects{N}{\vecpart(\vct{i})}$ and thus by \cref{lem:symstat}, we have $\Expect\left[\prod_{j \in \{1, \ldots m\}} Y_{i_j, j}\right] = \mu_{\vecpart(\vct{i})}\left(\vct{X}\right).$
Thus, the plug-in estimator of a moment polynomial is itself a moment polynomials of the same order. 
In other words, the moment polynomials are an invariant subspace of the linear operator $\rS$.

\begin{theorem}
  \label{thm:bootstrap_moment_poly}
  The subspace $\vspan\left(\left\{\mu_{\pi}\right\}_{\pi \in \Pi(m)}\right) \subset \R^{\Probs\left(\R^m\right)}$ is closed under the action of $\rS$. 
  On this subspace, $\rS$ takes the form
  \begin{equation}
    \rS \mu_{\sigma} = \sum \limits_{\pi \geq \sigma}\frac{\# \allrespects{N}{\pi}}{\# \allconstant{N}{\sigma}} \mu_{\pi}.
  \end{equation}
  In particular, the action of $\rS$ on the moment polynomials is represented by multiplying the coefficient vector $\vct{f} \in \R^{\Pi(m)}$ with the matrix $\mtx{S} \in \R^{\Pi(m) \times \Pi(m)}$ defined by 
  \begin{equation}
    S_{\pi, \sigma} = 
    \begin{cases}
      \frac{\# \allrespects{N}{\pi}}{\# \allconstant{N}{\sigma}} \quad &\text{for} \quad \sigma \leq \pi, \\
      0, \quad &\text{else}
    \end{cases}.
  \end{equation}
\end{theorem}
\begin{proof}
  We compute 
  \begin{align*}
    &\left(\rS \mu_\sigma\right)\left(\mtx{X}\right) 
    =     \Expect_{\mtx{Y} \sim_N \mtx{X}}\left[\mu_{\sigma}(\mtx{Y})\right] = \Expect \left[\frac{1}{\# \allconstant{N}{\sigma}}\sum_{\mtx{i} \in \allconstant{N}{\sigma}} \prod \limits_{j \in \{1, \ldots, m\}} Y_{i_j j}\right]\\
    =&     \Expect \left[\frac{1}{\# \allconstant{N}{\sigma}} \sum_{\pi \geq \sigma} \sum_{\mtx{i} \in \allrespects{N}{\pi}} \prod \limits_{j \in \{1, \ldots, m\}} Y_{i_j j}\right]
    =   \sum_{\pi \geq \sigma} \frac{\allrespects{N}{\pi}}{\# \allconstant{N}{\sigma}}  \Expect \left[ \frac{1}{\# \allrespects{N}{\pi}}\sum_{\mtx{i} \in \allrespects{N}{\pi}} \prod \limits_{j \in \{1, \ldots, m\}} Y_{i_j j}\right]\\
    & = \sum_{\pi \geq \sigma} \frac{\allrespects{N}{\pi}}{\# \allconstant{N}{\sigma}}  \mu_{\pi}\left(\mtx{X}\right)
  \end{align*}
  Here, the second equality follows from \cref{lem:plug-in}, the third one from \cref{lem:constantrespect} and the last one from \cref{lem:symstat}.
\end{proof}

\subsubsection*{The bootstrap computes M{\"o}bius inversion} We are now ready to deduce the relationship between resampling (and thus, the bootstrap) and M{\"o}bius inversion.
\begin{corollary}
  Defining the diagonal matrices $\mtx{R}, \mtx{C} \in \R^{\Pi(m) \times \Pi(m)}$ as 
  \begin{equation}
    R_{\pi, \sigma} = 
    \begin{cases}
      \# \allrespects{N}{\pi}, \quad &\text{for} \quad \pi = \sigma \\
      0, \quad &\text{else} 
    \end{cases}
    \quad
    \text{and}
    \quad 
    C_{\pi, \sigma} = 
    \begin{cases}
      \# \allconstant{N}{\pi}, \quad &\text{for} \quad \pi = \sigma\\
      0, \quad &\text{else} 
    \end{cases}.
  \end{equation}
  we have $\mtx{S} = \mtx{C}^{-1} \overline{\mtx{S}} \mtx{R}$, for $\overline{\mtx{S}}$ the inverse M{\"o}bius matrix defined in \cref{eq:mobius_matrix}.
  Up to diagonal scaling, the bootstrap computes M{\"o}bius inversion on $\Pi(m)$. 
\end{corollary}
As noted by \cite{koltchinskii2020asymptotically}, $\rS$ is the adjoint of the transition kernel of the Markov chain on $\Probs(\R^m)$ with transitions given by $\sim_N$.
Thus, its matrix representative $\mtx{S}$ is left stochastic. 
The operation $\vct{f} \mapsto \mtx{S} \vct{f}$ amounts to $\vct{f}$ percolating along the partition lattice, from finer to coarser partitions. 

\begin{figure}
  \label{fig:bootstrap_moebius}
  \centering
  \newcommand{\elm}[3]{\!\!\!\!\!\color{#1}\frac{\frac{N!}{(N - #2)!}}{N^#3}\!\!\!\!\!}
\begin{tikzpicture}

    \begin{scope}[yshift=-0.0cm, xshift=1.5cm]
    \node {$\underbrace{\tiny \begin{pmatrix} 
                \!\!\! \ddots \!\!\!\! & \cdots & \cdots & \cdots & \cdots & \cdots & \cdots & \cdots & \cdots & \cdots \\
                \!\!\! \vdots \!\!\!\! & \elm{orange}{3}{3} & 0 & 0 & 0 & 0 & 0 & 0 & 0 & 0 \\
                \!\!\! \vdots \!\!\!\! & \elm{orange}{2}{3}  & \elm{black}{2}{2}  & 0 & 0 & 0 & 0 & 0 & 0 & 0 \\
                \!\!\! \vdots \!\!\!\! & 0 & 0 & \elm{darksky}{2}{2}  & 0 & 0 & 0 & 0 & 0 & 0 \\
                \!\!\! \vdots \!\!\!\! & 0 & 0 & 0 & \elm{black}{2}{2}  & 0 & 0 & 0 & 0 & 0 \\
                \!\!\! \vdots \!\!\!\! & 0 & 0 & 0 & 0 & \elm{seagreen}{2}{2}  & 0 & 0 & 0 & 0 \\
                \!\!\! \vdots \!\!\!\! & 0 & 0 & 0 & 0 & 0 & \elm{rust}{2}{2}  & 0 & 0 & 0 \\
                \!\!\! \vdots \!\!\!\! & \elm{orange}{2}{3}  & 0 & 0 & 0 & 0 & 0 & \elm{black}{2}{2}  & 0 & 0 \\
                \!\!\! \vdots \!\!\!\! & \elm{orange}{2}{3}  & 0 & 0 & 0 & 0 & 0 & 1 & \elm{black}{2}{2}  & 0 \\
                \!\!\! \vdots \!\!\!\! & \elm{orange}{1}{3}  & \elm{black}{1}{2} & \elm{darksky}{1}{2} & \elm{black}{1}{2} & \elm{seagreen}{1}{2} & \elm{rust}{1}{2} & \elm{black}{1}{2} & \elm{black}{1}{2} & \elm{black}{1}{1}\,\, \, \,    
            \end{pmatrix}}_{\mtx{S} = }$};        
    \end{scope}

    \begin{scope}[xshift=9.15cm]
        \node (1 2 3 4) at(0, 3.0) {$1|2|3|4$};
        \node (13 2 4)  at(-2.75, 1.0) {$13|2|4$};
        \node (24 1 3)  at(-1.65, 1.0) {$24|1|3$};
        \node (12 3 4)  at(-0.55, 1.0) {$12|3|4$};
        \node (34 1 2)  at( 0.55, 1.0) {$34|1|2$};
        \node (14 2 3)  at( 1.65, 1.0) {$14|2|3$};
        \node (23 1 4)  at( 2.75, 1.0) {\color{orange} $23|1|4$};

        \node (123 4)  at(-2.85, -1.0) {\color{orange} $123|4$};
        \node (13 24)  at(-1.9, -1.0) {\color{darksky} $13|24$};
        \node (124 3)  at(-0.95, -1.0) {$124|3$};
        \node (12 34)  at( 0.0, -1.0) {\color{seagreen} $12|34$};
        \node (134 2)  at( 0.95, -1.0) {\color{rust} $134|2$};
        \node (14 23)  at( 1.9, -1.0) {\color{orange} $14|23$};
        \node (234 1)  at( 2.85, -1.0) {\color{orange} $234|1$};

        \node (1234) at(0, -3.0) {$\color{orange}1\color{darksky}2\color{seagreen}3\color{rust}4$};

        \draw [thick] (1 2 3 4.south) -- (13 2 4.north);
        \draw [thick] (1 2 3 4.south) -- (24 1 3.north);
        \draw [thick] (1 2 3 4.south) -- (12 3 4.north);
        \draw [thick] (1 2 3 4.south) -- (34 1 2.north);
        \draw [thick] (1 2 3 4.south) -- (14 2 3.north);
        \draw [thick] (1 2 3 4.south) -- (23 1 4.north);

        \draw [thick] (13 2 4.south) -- (123 4.north);
        \draw [thick] (13 2 4.south) -- (13 24.north);
        \draw [thick] (13 2 4.south) -- (134 2.north);
        \draw [thick] (24 1 3.south) -- (13 24.north);
        \draw [thick] (24 1 3.south) -- (124 3.north);
        \draw [thick] (24 1 3.south) -- (234 1.north);
        \draw [thick] (12 3 4.south) -- (123 4.north);
        \draw [thick] (12 3 4.south) -- (124 3.north);
        \draw [thick] (12 3 4.south) -- (12 34.north);
        \draw [thick] (34 1 2.south) -- (12 34.north);
        \draw [thick] (34 1 2.south) -- (134 2.north);
        \draw [thick] (34 1 2.south) -- (234 1.north);
        \draw [thick] (14 2 3.south) -- (124 3.north);
        \draw [thick] (14 2 3.south) -- (134 2.north);
        \draw [thick] (14 2 3.south) -- (14 23.north);
        \draw [ultra thick, color=orange] (23 1 4.south) -- (123 4.north);
        \draw [ultra thick, color=orange] (23 1 4.south) -- (14 23.north);
        \draw [ultra thick, color=orange] (23 1 4.south) -- (234 1.north);

        \draw [ultra thick, color=orange] (123 4.south) -- (1234.north);
        \draw [ultra thick, color=darksky] (13 24.south) -- (1234.north);
        \draw [thick] (124 3.south) -- (1234.north);
        \draw [ultra thick, color=seagreen] (12 34.south) -- (1234.north);
        \draw [ultra thick, color=rust] (134 2.south) -- (1234.north);
        \draw [ultra thick, color=orange] (14 23.south) -- (1234.north);
        \draw [ultra thick, color=orange] (234 1.south) -- (1234.north);

    \end{scope}
\end{tikzpicture}
  \caption{\textbf{The bootstrap and M{\"o}bius inversion.} We show (a lower-right submatrix of) $\mtx{S}$ for $m = 4$. 
  Color coding identifies its columns with intervals on the partition lattice $\Pi(4)$. 
  Multiplying $\mtx{S}$ to a vector results in its nonzero values percolating down the Hasse diagram (drawn in the finest-on-top orientation, c.f. \cref{fig:moebius}).}
\end{figure}

\subsubsection*{Reduction} 
The cardinality of the sets $\Pi(m)$ is given by the Bell numbers, growing superexponentially with $m$.
A reduced representation arises by tracking not the coefficients $f_{\pi}$, but their level-wise sums $(f_{i})_{1 \leq i \leq m}, f_{i} \coloneqq \sum_{\pi \in \Pi(m): \# \pi = i}$.
For $\vct{e}_{\pi}$ the standard basis vector corresponding to $\pi$ and $1 \leq l \leq m$, define 
\begin{equation}
  \check{\vct{e}}_{l} \coloneqq \sum \limits_{\pi \in \Pi(m): \# \pi = l} \vct{e}_{\pi}, \quad \quad \quad \check{\mtx{E}} = [\check{\vct{e}}_{1}, \ldots, \check{\vct{e}}_{m}], \quad \quad \quad \check{\mtx{S}} \coloneqq \check{\mtx{E}}^\tp \mtx{S} \check{\mtx{E}}.
\end{equation}

\begin{theorem}
  We have, for all $\vct{f} \in \R^{\Pi(m)}$, $\check{\mtx{E}}^\tp \mtx{S} \vct{f} = \check{\mtx{S}} \check{\mtx{E}}^{\tp}\vct{f}$.   
  In other words, the action of $\rS$ on moment polynomials induces a linear map $\check{\mtx{S}}$ on coefficient-sums over partitions of the same size.
  The upper triangular matrix $\check{\mtx{S}}$ has entries given by 
  \begin{equation}
    \check{S}_{ij} = 
    \begin{cases}
      \stirling{j}{i} \frac{\frac{N!}{(N - i)!}}{N^j} \quad &\text{for} \quad i \leq j, \\
      0, \quad &\text{else},
    \end{cases}
  \end{equation} 
  for $\stirling{j}{i}$ the Stirling number of the second kind (number of size-$i$ partitions of $j$ objects). 
\end{theorem}
\begin{proof}
  For every partition $\pi$ with $\# \pi = j$ and $\sigma$ with $\# \sigma = i$, we have $S_{\pi, \sigma} = \frac{\frac{N!}{(N - j)!}}{N^j}$ if $\pi \leq \sigma$, and $S_{\pi, \sigma} = 0,$ otherwise. 
  For $i > j$, we can never have $\pi \leq \sigma$, hence $\check{S}_{ij} = 0$.
  For $i \leq j$, the key observation is that for every partition $\pi$ with $\# \pi = j$, there are $\stirling{j}{i}$ partitions $\sigma$ with $\# \sigma = i$ such that $\sigma \geq \pi$.
  This is because the coarsenings of $\pi$ of cardinality $i$ are in bijection with the partitions $\pi$ into $i$ nonempty parts.
  Thus, the summary transfer of coefficients from partitions of cardinalities $j$ and $i$ level the action of $\mtx{S}$ is independent of how the coefficients are distributed among size $j$ and $i$ partitions, and the resulting coefficient is $\stirling{j}{i} \frac{N!}{(N - j)!} N^{-i}$, yielding the result. 
\end{proof}

\section{Consequences for the iterated bootstrap} 

\subsection{Convergence rates on moment polynomials}
\label{sec:moment_poly_convergence}

\subsubsection*{Inverting $\mtx{S}$ iteratively} 
As discussed in \cref{sec:bootstrap}, the bootstrap amounts to using a Neumann series to invert the operator $\rS$.
In the case of $F$ being a moment polynomial, this means inverting the matrix $\mtx{S}$ using a Neumann series or, equivalently, Richardson iteration,  
\begin{equation}
  \vct{f}^{(k + 1)} = \vct{f}^{(k)} + \left( \vct{f} - \mtx{S} \vct{f}^{(k)} \right), \quad \text{with} \quad \mtx{f}^{(0)} = \mtx{f}. 
\end{equation}
Following standard error analysis for iterative methods \cite{saad2003iterative}, we obtain 
\begin{equation}
  \label{eqn:rich_error}
  \mtx{S} \vct{f}^{(k + 1)} - \vct{f} = \left( \Id - \mtx{S} \right)(\mtx{S} \vct{f}^{(k)} - \vct{f}). 
\end{equation}
Thus, the bias reduction of the iterated bootstrap amounts to the contractivity of $(\Id - \mtx{S})$. 

\subsubsection*{Contraction property}
We now analyze the contraction property of $\mtx{S}$ in the $1$-norm.  
To this end, we need the following lemma.
\begin{lemma}
  \label{lem:colsum}
  $\mtx{S}$ is entrywise nonnegative and its column sums satisfy 
  \begin{equation}
    \sum \limits_{\sigma \in \Pi(m)} S_{\sigma \pi} = 
    \begin{cases}
      1, \quad &\text{for} \quad \# \pi \leq N \\
      0, \quad &\text{else}.
    \end{cases}
  \end{equation}
\end{lemma}
\begin{proof}
  The nonnegativity of the entries follows directly from the definition of $\mtx{S}$.
  Unity of its column sums for $\# \pi \leq N$ follows from \cref{lem:constantrespect}. 
  For $\# \pi > N$, they are zero since $\allrespects{N}{\pi}$ is empty.  
\end{proof}

\begin{lemma}
  \label{lem:opbound}
  The operator $1$-norm of $(\Id - \mtx{S})$ is
  \begin{equation}
    \|\Id - \mtx{S}\|_1 =  
    \begin{cases}
      2 \left( 1 - \frac{\frac{N!}{\left(N- m \right)!}}{N^{m}}\right)
, \quad &\text{for} \quad m \leq N \\
      2, \quad &\text{else}.
    \end{cases}
  \end{equation}
\end{lemma}
\begin{proof}
  For a unit vector $\univec_{\pi}$ with $\# \pi \leq N$, we have
  \begin{equation}
    \left\|\left(\Id - \mtx{S}\right)\univec_{\pi} \right\|_1 
    = \left|1 - \mtx{S}_{\pi, \pi}\right| + \left|\sum \limits_{\pi \leq \tilde{\pi}} \mtx{S}_{\tilde{\pi}, \pi} \right| 
    = 2  \left|1 - \mtx{S}_{\pi, \pi}\right| 
    = 2 \left( 1 - \frac{\frac{N!}{\left(N- \#\pi \right)!}}{N^{\#\pi}}\right).
  \end{equation}
  For $\# \pi > m$  we have $\left(\Id - \mtx{S}\right) \univec_\pi = \univec_\pi$.
  Here, we have used \cref{lem:colsum}. Thus, we obtain
  \begin{equation}
    \left\|\left(\Id - \mtx{S}\right)\sum \limits_{\pi \in \Pi(m)} g_{\pi} \univec_{\pi} \right\|_1 
    \leq \sum \limits_{\pi \in \Pi(m)} \left|g_{\pi}\right| \left\| \left(\Id - \mtx{S}\right)\univec_{\pi} \right\| \leq 2 \left( 1 - \frac{\frac{N!}{\left(N- m \right)!}}{N^{m}}\right) \sum \limits_{\pi \in \Pi(m)} \left|g_{\pi}\right|.
  \end{equation}
  The last inequality follows from observing that
  $k \mapsto \frac{N!}{(N - k)!} N^{-k}$ is decreasing in $k$ for $N \geq k$.
\end{proof}

\subsection*{The convergence rate of the iterated bootstrap}
The contraction property of $\Id - \mtx{S}$ yields a convergence result for the bias of the iterated bootstrap.
This estimate forms the basis of the convergence results throughout this work.
\begin{theorem}
  \label{thm:bias_convergence}
  Applying the $k$-times iterated bootstrap to $F(\mtx{X}) = \vct{\mu}\left(\mtx{X}\right)^\tp \vct{f}$ has bias  
  \begin{equation}
    \left|\Bs^{k} \rS \Func(\mtx{X}) - \Func(\mtx{X})\right| \leq \|\vct{\mu}\|_{\infty} \left(2 \left(1 - \frac{\frac{N!}{\left(N- m \right) !}}{N^{m}} \right) \right).
  \end{equation} 
\end{theorem} 
\begin{proof}
  We write
  \begin{equation}
    \left|\Bs^{k} \rS \Func(\mtx{X}) - \Func(\mtx{X})\right| = \left|\vct{\mu}^{\tp} \left(\mtx{S} \vct{f}^{(k)} - \vct{f}\right)\right| = \left|\vct{\mu}^{\tp} \left(\Id - \mtx{S} \right)^{k + 1} \vct{f}\right| \leq  2 \|\vct{\mu}\|_{\infty} \left(1 - \frac{\frac{N!}{\left(N- m \right) !}}{N^{m}} \right).
  \end{equation}
  Here, we have used \cref{lem:opbound} and Equation~\cref{eqn:rich_error} for the error of Richardson iteration.
\end{proof}

\subsubsection*{Linear convergence of the bootstrap} In many asymptotic settings, the $k$-times iterated bootstrap is known to improve the convergence rate of $\rS\Func\left(\mtx{X}\right) \xrightarrow[]{N \rightarrow \infty}\Func(\mtx{X})$ from $N^{-1}$ to $N^{-(k + 1)}$ \cite{hall2013bootstrap}. 
For fixed but large $N$, this suggests the linear convergence of the iterated bootstrap.  
But the convergence rate in \cref{thm:bias_convergence} depends not only on $N$ but also on the order $m$ of the moment polynomial. 
This begs the question of how fast $N$ needs to grow as a function of $m$ to ensure uniform linear convergence of the iterated bootstrap?
The machinery developed so far enables a precise characterization of the regime of linear convergence.

\begin{theorem}
  \label{thm:linconvergence}
  For $N \geq \max(\alpha m^2, m + 1)$, we have
  \begin{equation}
    \frac{\frac{N!}{\left(N- m \right) !}}{N^{m}} \geq \exp(-4^{-1}) \exp(-1/\alpha).
  \end{equation} 
  In particular, by setting $\alpha = 8 $ we obtain
  \begin{equation}
    \frac{\frac{N!}{\left(N- m \right) !}}{N^{m}} \geq \exp(-1/4) \geq 3/4,  
  \end{equation}
  thus achieving a bound $\|\Id - \mtx{S}\|_1 \leq 1/2$, independently of $m$.
  Conversely, if $N$ as a function of $m$ is such that $N/m^2 \rightarrow 0$, we have $\|\Id - \mtx{S}\|_1 \rightarrow 1$, ruling out uniform linear convergence. 
\end{theorem}

\begin{proof}
  We use the bound on the factorial due to \cite{robbins1955remark}, of the form
  \begin{equation}
    \sqrt{2\pi N} \left(\frac{N}{\econst}\right)^N \exp\left(\frac{1}{12 N + 1}\right) \leq N! \leq \sqrt{2\pi N} \left(\frac{N}{\econst}\right)^N \exp\left(\frac{1}{12 N}\right).
  \end{equation}
  Defining the uniformly upper and lower bounded quantities $c, C$ as
  \begin{equation}
  c \coloneqq \exp\left(\frac{1}{12N + 1} - \frac{1}{12(N - m)}\right), \quad C \coloneqq \exp\left(\frac{1}{12N} - \frac{1}{12(N - m) + 1}\right),
  \end{equation}
  we can ensure, for $N \geq m + 1$, that $\left(\frac{N!}{(N - m)!} \right) / N^m \in [\gamma c, \gamma C]$ by choosing
  \begin{equation}
    \gamma = \sqrt{\frac{N}{N - m}} \frac{ \left(\frac{N}{\econst}\right)^N}{N^m  \left(\frac{N - m}{\econst}\right)^{N - m}} = \sqrt{\frac{N}{N - m}} \econst^{-m}  \left(\frac{N}{N - m}\right)^{N - m} \!\!  
    =  \econst^{-m}  \left(1 + \frac{m}{N - m}\right)^{N - m + \frac{1}{2}}.
  \end{equation}
  To attain uniform exponential convergence of the bootstrap, we need to choose $N$ as a function of $m$ such that $\gamma$ is lower bounded by a constant. 
  For $N = \alpha m^2$, we have 
  \begin{equation}
    \gamma 
    = \econst^{-m}  \left(\frac{\alpha m^2}{\alpha m^2 - m}\right)^{\alpha m^2 - m + \frac{1}{2}} 
    = \econst^{-m}  \left(1 + \frac{m}{\alpha m^2 - m}\right)^{\alpha m^2 - m + \frac{1}{2}}. 
  \end{equation}
  We can rewrite $\gamma$ as 
  \begin{equation}
    \gamma = \exp\left(- m + \left(\alpha m^2 - m + \frac{1}{2}\right) \log\left(1 + \frac{m}{\alpha m^2 - m}\right)\right).
  \end{equation}
  Now, using the lower bound $x / (1 + x) \leq \log(1 + x)$ on the logarithm, we obtain
  \begin{equation}
    \gamma 
    \geq \exp\left(- m + (\alpha m^2 - m) \frac{\frac{m}{\alpha m^2 - m}}{1 + \frac{m}{\alpha m^2 - m}}\right) 
    = \exp\left(- 1/\alpha\right).
  \end{equation}
  Since $n \mapsto \frac{n!}{(n - m)!} / n^m$ is increasing, and $N/(N - m) > 1$, this yields the desired linear convergence for all $N \geq \max(\alpha m^2, m + 1)$.
  The loose estimate $c \geq \exp(-1/8)$ yields linear convergence.
  The converse follows from the upper bound $\log(1 + x) \leq x/\sqrt{1 + x}$ for $x \geq 0$,
  \begin{equation}
    \gamma \leq \exp\left(- m +  \frac{\left(\alpha m^2 - m + \frac{1}{2}\right)\frac{m}{\alpha m^2 - m}}{\sqrt{1 + \frac{m}{\alpha m^2 - m}}}\right)
    = \exp\left(- m + \frac{m}{\sqrt{1 + \frac{m}{\alpha m^2 - m}}} + \frac{\frac{\frac{m}{\alpha m^2 - m}}{\sqrt{1 + \frac{m}{\alpha m^2 - m}}}}{2}\right).
  \end{equation}
  For $m$ large enough, the last summand is bounded by $1/m$. By taking logarithms we obtain
  \begin{equation}
    \log(\gamma) \leq \frac{m\left(1 - \sqrt{1 + \frac{m}{\alpha m^2 - m}}\right)}{\sqrt{1 + \frac{m}{\alpha m^2 - m}}} + \frac{1}{m} 
    \leq \frac{m\left(1 - \sqrt{1 + \frac{m}{\alpha m^2}}\right)}{\sqrt{1 + \frac{m}{\alpha m^2}}} + \frac{1}{m}
    \xrightarrow[]{m \rightarrow \infty}  - \frac{1}{2 \alpha}. 
  \end{equation}
  If $N$ does not grow at least as $m^2$, $\alpha$ can become arbitrarily small. 
  The iterated bootstrap does not converge with linear rate independent of $m$ unless $N$ grows at least as with $m^2$.
\end{proof}

\subsubsection*{Sublinear convergence}
\Cref{thm:linconvergence} sharply characterizes the uniform linear convergence of the iterated bootstrap.   
But the iterated bootstrap is known to converge to an unbiased estimator for any fixed $N \geq m$ \cite{hall2013bootstrap}.
We reprove this using our linear algebraic perspective.
\begin{theorem}  
  For any $N \geq m$, we have $\lim \limits_{k \rightarrow \infty} (\Id - \mtx{S})^{k} = 0$. Thus, the iterated bootstrap converges to an unbiased estimator.
\end{theorem}
\begin{proof}
  The diagonal values of $\mtx{S}$ are contained in $(0,1)$. 
  Thus, so are the diagonal values of $\Id - \mtx{S}$.  
  By triangularity of $\Id - \mtx{S}$ these are also the eigenvalues of $\Id - \mtx{S}$ and the latter thus has spectral radius smaller than one, from which the result follows.
\end{proof}

\subsection*{General moment polynomials} Our analysis applies to moment polynomials of order one in each variable.
It extends to general monomials by duplicating variables. 
For instance, the moment polynomial $F(X) = \mathbb{E}\left[x_1^2\right]\mathbb{E}\left[x_2^3\right]$ is equal to $\hat{F}(X) = \mathbb{E}\left[x_1 x_3\right]\mathbb{E}\left[x_2 x_4 x_5\right]$ applied to the five-variate random variable $(x_1, x_1, x_2, x_2, x_2)$.
To incorporate sums of monomials, and thus general monomials, we observe that the union of partition lattices is a partially ordered set when treating the different lattices as independent. 
By linearity and block-diagonality (with respect to the distinct lattices) of the sampling operator, the bootstrap solves M{\"o}bius inversion on this this partially ordered set. 
As illustrated in \cref{fig:bootstrap_moebius_duplicates}, this allows us to incorporate general moment polynomials in our framework. 
The convergence rate of the iterated bootstrap is determined by the largest eigenvalue of the matrix representative of $\rS$. 
Thus, it only depends on the order of the larges monomial present in the moment polynomials, and the analysis of the previous section applies to arbitrary moment polynomials.

\begin{figure}
  \label{fig:bootstrap_moebius_duplicates}
  \centering
  \begin{tikzpicture}
    \begin{scope}[xshift=6.0cm, yshift=3.5cm]
        \node () at(0, 0) {$F(\mtx{X}) = {\color{joshua}\Expect_{\vct{x} \sim \mtx{X}}\left[x_1^2\right]} + {\color{orange} \Expect_{\vct{x} \sim \mtx{X}}\left[x_1x_3\right]\Expect_{\vct{x} \sim \mtx{X}}\left[x_1\right]} + {\color{rust}\Expect_{\vct{x} \sim \mtx{X}}\left[x_1x_2^2\right]\Expect_{\vct{x} \sim \mtx{X}}\left[x_2\right]}$}; 
    \end{scope}

    \begin{scope}[xshift=-3.5cm, yshift=-0.0cm]
        \node (1 2) at(1, 0.5) {$x_1|x_1$};
        \node (12) at(1, -0.5) {$\color{joshua}x_1 x_1$};

        \draw [thick] (1 2) -- (12);
    \end{scope}

    \begin{scope}[xshift=-2.5cm, yshift=2.5cm]
        \node (1 2 3) at(1.5, 1) {$x_1|x_1|x_3$};
        \node (12 3)  at(0, 0) {$x_1 x_1|x_3$};
        \node (13 2)  at(1.5, 0) {$\color{orange} x_1 x_3|x_1$};
        \node (23 1)  at(3, 0) {$\color{orange} x_1 x_3|x_1$};
        \node (123)   at(1.5, -1) {$x_1 x_1 x_3$};

        \draw [thick] (1 2 3.south) -- (12 3.north);
        \draw [thick] (1 2 3.south) -- (13 2.north);
        \draw [thick] (1 2 3.south) -- (23 1.north);

        \draw [thick]  (12 3.south) -- (123.north);
        \draw [thick]  (13 2.south) -- (123.north);
        \draw [thick]  (23 1.south) -- (123.north);

    \end{scope}

    \begin{scope}[xshift=5.25cm]
        \node (1 2 3 4) at(0, 2.25) {$x_1|x_2|x_2|x_2$};
        \node (13 2 4)  at(-5.0, 0.75) {$x_1x_2|x_2|x_2$};
        \node (24 1 3)  at(-3.0, 0.75) {$x_2 x_2|x_1|x_2$};
        \node (12 3 4)  at(-1.0, 0.75) {$x_1 x_2|x_2|x_2$};
        \node (34 1 2)  at( 1.0, 0.75) {$x_2 x_2|x_1|x_2$};
        \node (14 2 3)  at( 3.0, 0.75) {$x_1 x_2|x_2|x_2$};
        \node (23 1 4)  at( 5.0, 0.75) {$x_2 x_2|x_1|x_2$};

        \node (123 4)  at(-6.0, -0.75) {$\color{rust}x_1x_2x_2|x_2$};
        \node (13 24)  at(-4.0, -0.75) {$x_1x_2|x_2x_2$};
        \node (124 3)  at(-2.0, -0.75) {$\color{rust} x_1x_2x_2|x_2$};
        \node (12 34)  at( 0.0, -0.75) {$x_1x_2|x_2x_2$};
        \node (134 2)  at( 2.0, -0.75) {$\color{rust} x_1x_2x_2|x_2$};
        \node (14 23)  at( 4.0, -0.75) {$x_1x_2|x_2x_2$};
        \node (234 1)  at( 6.0, -0.75) {$x_2x_2x_2|x_1$};

        \node (1234) at(0, -2.25) {$x_1x_2x_2x_2$};

        \draw [thick] (1 2 3 4.south) -- (13 2 4.north);
        \draw [thick] (1 2 3 4.south) -- (24 1 3.north);
        \draw [thick] (1 2 3 4.south) -- (12 3 4.north);
        \draw [thick] (1 2 3 4.south) -- (34 1 2.north);
        \draw [thick] (1 2 3 4.south) -- (14 2 3.north);
        \draw [thick] (1 2 3 4.south) -- (23 1 4.north);

        \draw [thick] (13 2 4.south) -- (123 4.north);
        \draw [thick] (13 2 4.south) -- (13 24.north);
        \draw [thick] (13 2 4.south) -- (134 2.north);
        \draw [thick] (24 1 3.south) -- (13 24.north);
        \draw [thick] (24 1 3.south) -- (124 3.north);
        \draw [thick] (24 1 3.south) -- (234 1.north);
        \draw [thick] (12 3 4.south) -- (123 4.north);
        \draw [thick] (12 3 4.south) -- (124 3.north);
        \draw [thick] (12 3 4.south) -- (12 34.north);
        \draw [thick] (34 1 2.south) -- (12 34.north);
        \draw [thick] (34 1 2.south) -- (134 2.north);
        \draw [thick] (34 1 2.south) -- (234 1.north);
        \draw [thick] (14 2 3.south) -- (124 3.north);
        \draw [thick] (14 2 3.south) -- (134 2.north);
        \draw [thick] (14 2 3.south) -- (14 23.north);
        \draw [thick] (23 1 4.south) -- (123 4.north);
        \draw [thick] (23 1 4.south) -- (14 23.north);
        \draw [thick] (23 1 4.south) -- (234 1.north);

        \draw [thick] (123 4.south) -- (1234.north);
        \draw [thick] (13 24.south) -- (1234.north);
        \draw [thick] (124 3.south) -- (1234.north);
        \draw [thick] (12 34.south) -- (1234.north);
        \draw [thick] (134 2.south) -- (1234.north);
        \draw [thick] (14 23.south) -- (1234.north);
        \draw [thick] (234 1.south) -- (1234.north);
    \end{scope}
\end{tikzpicture}
  \caption{\textbf{General moment polynomials} fit into our framework by considering M{\"o}bius inversion on partially ordered sets obtained as unions of partition lattices of different order, with possibly duplicated variables. As explained in the caption of \cref{fig:moebius}, we draw the Hasse diagram in the finest-on-top orientation.}
\end{figure}

\begin{corollary}
  Let $F$ be an arbitrary moment polynomial of order $m$ in $d$ variables. 
  Then, the following dichotomy holds.
  For $N \geq \max(8 m^2, m + 1)$, we have $\|\Id - \mtx{S}\|_1 \leq 1/2$ independently of $m$. Thus, the iterated bootstrap converges linearly.
  Conversely, if $N$ as a function of $m$ satisfies $N/m^2 \rightarrow 0$, we have $\|\Id - \mtx{S}\|_1 \rightarrow 1$, ruling out uniform linear convergence. 
\end{corollary}

\begin{proof}
  As illustrated in \cref{fig:bootstrap_moebius_duplicates}, we can view general moment polynomials as having coefficients on a union of partition lattices. 
  The matrix $\mtx{S}$ is block diagonal with blocks that satisfy \cref{thm:linconvergence}.  
  Thus, the 1-norm of $\Id - \mtx{S}$ is the maximum of the 1-norms of the blocks, which are determined by \cref{thm:linconvergence}.
\end{proof}

\subsection{Beyond moment polynomials} 

\subsubsection*{The set $\Pi(d, m)$}
As illustrated in the last section, our results on M{\"o}bius inversion and the iterated bootstrap apply to moment polynomials of order $m$ in $d$ variables $m$ of $d$-variate distributions by selecting (possibly with repetitions) $m$ variates of the probability distributions.
To this end, we define the set $\Pi(d, m)$ as the set of multiset partitions of index multisets of size $m$ chosen with replacement from $\{1, \ldots, d\}$.
For a partition $\pi \in \Pi(d, m)$, we write $\mu_{\pi}(\mtx{X})$ to denote its moment product. 
For instance, defining $\pi \coloneqq [[1, 1, 5,], [1, 6], [1, 6]] \in \Pi(8, 7)$ we have $\mu_{\pi} = \Expect_{\vct{x} \sim \mtx{X}}\left[x_{1} x_{1} x_{5}\right]\Expect_{\vct{x} \sim \mtx{X}}\left[x_{1} x_{6}\right]\Expect_{\vct{x} \sim \mtx{X}}\left[x_{1} x_{6}\right]$. 
Note that we use the $[\cdot]$ instead of $\{\cdot\}$ to denote the partitions in $\Pi(d, m)$, as the latter are multisets.
We formally define $\Pi(d, m)$ as
\begin{equation}
\left\{[s_1, \ldots, s_i] \middle| i \in \N, (s_j)_{1 \leq j \leq i} \text{ are multisets with elements in } \{1, \ldots, d\}, \sum_{j=1}^i \#s_j = m\right\}.
\end{equation}

\subsubsection*{Extension by approximation} We use polynomial approximation to derive nonasymptotic convergence results for functionals $F$ that are not moment polynomials.
Consider $\Func$ given by
\begin{equation}
  F(\mtx{X}) = \sum \limits_{m = 0}^\infty \sum_{\pi \in \Pi(m, d)} \mu_{\pi}(\mtx{X}) f_\pi.
\end{equation}
According to \cref{thm:linconvergence}, the bias in terms with $m \lessapprox \sqrt{N}$ is attenuated at a linear rate. 
Meanwhile, according to \cref{lem:opbound}, the error for $m \gtrapprox \sqrt{N}$ could increase by a constant factor with each bootstrap iteration. 
The decay of the product $\mu_{\pi} f_{\pi}$ for $\pi \in \Pi(d, m)$  thus determines the optimal number of bootstrap iterations for a given $N$ and the bias reduction that they can achieve. 
Prior work observed that the bootstrap's success depends on smoothness of $\Func$ and the tail decay of $\mtx{X}^0$ \cite{hall2013bootstrap}.
The former implies the decay of $\vct{f}$, and the latter bounds the growth of $\vct{\mu}$. 
We now provide quantitative finite sample results supporting this intuition.

\begin{theorem}
  \label{thm:bias_reduction_general}
  Consider $F$ mapping $d$-variate random variables $\mtx{X}$ to $\R$ as 
  \begin{equation}
    F(\mtx{X}) = \sum \limits_{m = 0}^\infty \sum_{\pi \in \Pi(d, m)} \mu_{\pi}(\mtx{X}) f_\pi,
  \end{equation}
  and define $\alpha, \beta, \gamma \in \R^{\N_{\geq 0}}$ as
  \begin{equation}
    \alpha_{m} \coloneqq \max \limits_{\pi \in \Pi(d, m)} |\mu_{\pi}(\mtx{X})|, \quad \beta_{m} \coloneqq \sum \limits_{\pi \in \Pi(d, m)} |f_{\pi}|, \quad \gamma_m \coloneqq \sum \limits_{k = m}^\infty \alpha_m \beta_m.
  \end{equation}
  Then, the $N$-sample, $k = \left\lfloor \frac{1}{2} \log_2\left(\gamma_0\right) - \frac{1}{2}\log_2\left(\gamma_{\left\lceil \sqrt{\frac{N}{8}} \right\rceil}\right) \right\rfloor$-times iterated bootstrap has bias
  \begin{equation}
    \left|\left(\mathrm{Id} - \rS \right)^{k + 1}F(\mtx{X})\right| \leq \sqrt{16 \gamma_{0} \gamma_{\lceil \sqrt{N / 8}\rceil}}.
  \end{equation}
\end{theorem}
\begin{proof}
  We begin by splitting $F$ as 
  \begin{equation}
    F(\mtx{X})
    =  \sum \limits_{m = 0}^\infty \sum_{\pi \in \Pi(d, m)} \mu_{\pi}\left(\mtx{X}\right) f_\pi 
    = \sum \limits_{m = 0}^{\left\lceil \sqrt{\frac{N}{8}} \right \rceil - 1} \sum_{\pi \in \Pi(d, m)} \mu_{\pi}\left(\mtx{X}\right) f_\pi
    + \sum \limits_{m = \left\lceil \sqrt{\frac{N}{8}} \right\rceil}^{\infty} \sum_{\pi \in \Pi(d, m)} \mu_{\pi}\left(\mtx{X}\right) f_\pi.
  \end{equation} 
  Using \cref{thm:linconvergence} and \cref{lem:opbound}, we upper bound 
  \begin{equation}
  \left|\left(\mathrm{Id} - \rS\right)^{k + 1}F(\mtx{X})\right| \leq 2^{- (k + 1)} \sum \limits_{m = 0}^{\left\lceil \sqrt{\frac{N}{8}} \right \rceil - 1} \alpha_m \beta_m
    + 2^{k + 1} \sum \limits_{m = \left\lceil \sqrt{\frac{N}{8}} \right\rceil}^{\infty} \alpha_m \beta_m \leq \frac{\gamma_{0}}{2^{k + 1}} + 2^{k + 1} \gamma_{\lceil \sqrt{N / 8}\rceil}.
  \end{equation}
  Plugging $k$ into this estimate, we obtain the result 
  \begin{equation}
  \left|\left(\mathrm{Id} - \rS\right)^{k + 1}F(\mtx{X})\right| \leq \sqrt{4 \gamma_{0}  \gamma_{\lceil \sqrt{N / 8}\rceil}}.
  \end{equation}
\end{proof}
We now present two instantiations of the above theorem.
\subsubsection*{Superlinear convergence for bandlimited functions} 
We can specialize the above theorem to the case of bandlimited functions of expectations of sub-Gaussian random variables. 
\begin{corollary}
    Let $F(\mtx{X }) \coloneqq \Phi\left(\Expect_{\vct{x} \sim \mtx{X}}\left[ \vct{x}\right]\right)$ for a $\Phi:\R^d \longrightarrow \R$ that satisfies, for any multiindex $\alpha \in \N^d$, $\sup_{\vct{x} \in \R^d} |\cnst{D}^\alpha \phi| \leq C_\eta \exp\left(\eta |\alpha|\right)$. Let $\mtx{X} \in \Probs(d)$ be sub-Gaussian in the sense that its $k$-th moment is upper-bounded by $\theta^k k^{k / 2}$.
    For $N \geq 8(8d\theta + 1)^2$ samples and 
    \begin{equation}
      k = \left\lfloor \frac{1}{2}\left(\log(1 + 8d\theta \sqrt{4 d \theta}^{8d \theta }) - \sqrt{\frac{N}{8}}\log\left(4d\theta/\sqrt{\frac{N}{8}}\right)\right)\right\rfloor    
    \end{equation}
     bootstrap iterations, the bias is upper-bounded as  
    \begin{equation}
      \left|\left(\mathrm{Id} - \rS\right)^{k + 1}F(\mtx{X})\right|  \leq \sqrt{16}\left(1 + 8d\theta \sqrt{4 d \theta}^{8d \theta}\right)\left(\frac{4 d \theta}{\sqrt{\frac{N}{8}}}\right)^{\sqrt{\frac{N}{32}}}.
    \end{equation}
\end{corollary}
\begin{proof} 
  Using Taylor expansion, we can write
  \begin{equation}
    F\left(\mtx{X}\right) = \sum \limits_{m = 0}^{\infty} \sum \limits_{\substack{\pi \in \Pi(d, m)\\ \# \pi = m}}\frac{\cnst{D}^\pi\Phi(0)}{m!} \mu_{\pi}(\mtx{X}). 
  \end{equation}
  Here, we abuse notation slightly and define for $\pi \in \Pi(d, m)$ with $\# \pi = m$, $\cnst{D}^\pi \coloneqq \prod \limits_{\{i\} \in \pi}\partial_i$.
  By lower bounding $m!$ with $(m/4)^m$ and using $\# \{\pi \in \Pi(d, m): \# \pi = m\} = d^m$, we obtain 
  \begin{equation}
    \alpha_{m} \leq \frac{(4d)^m}{m^m}, \quad \beta_{m} \leq \theta^m m^{m / 2}, \quad \gamma_{m} = \sum \limits_{k = m}^{\infty} \sqrt{\left(\frac{4d\theta}{k} \right)^{k}} 
    \overset{m > 8d\theta}{\leq} \sqrt{\left(\frac{4d\theta}{m} \right)}^{m}.
  \end{equation} 
  To deploy \cref{thm:bias_reduction_general}, it suffices to bound $\gamma_{0}$ losely by $1 + 8d\theta \sqrt{4 d \theta}^{8d \theta }$.
  For $N \geq 8 (8 d \theta + 1)^2$, and $\frac{1}{2}\left(\log(1 + 8d\theta \sqrt{4 d \theta}^{8d \theta }) - \sqrt{\frac{N}{8}}\log\left(4d\theta/\sqrt{\frac{N}{8}}\right)\right)$ bootstrap iteration, we attain an upper bound for the bias of $\sqrt{16}\left(1 + 8d\theta \sqrt{4 d \theta}^{8d \theta}\right)\left(\frac{4 d \theta}{\sqrt{\frac{N}{8}}}\right)^{\sqrt{\frac{N}{32}}}.$  
\end{proof}

\subsubsection*{Inverse of average linear systems} We now consider a simple example involving the inverse of a mean of a random matrix.  

\begin{corollary}
  Let $\mtx{A} \in \R^{n \times n}$ a random symmetric and positive matrix with eigenvalues in $(0, \sigma) \subset (0, 1)$ almost surely.   
  Let $\mtx{X} \in \Probs\left(\R^{d \coloneqq n^2}\right)$ be the distribution of $\mtx{A}$ and define
  \begin{equation}
    F(\mtx{X}) \coloneqq \trace \left(\Id + \Expect_{\mtx{A}\sim \mtx{X}}\left[\mtx{A}\right] \right)^{-1}.
  \end{equation}
  The $k = \left\lfloor -\sqrt{\frac{N}{8}}\log\left(\sigma\right) \right\rfloor$ then attains bias 
  \begin{equation}
    \left|\left(\mathrm{Id} - \rS \right)^{k + 1}F(\mtx{X})\right| \leq \frac{4}{1 - \sigma}\sigma^{\sqrt{N / 32}}.
  \end{equation}
\end{corollary}
\begin{proof}
  Using Taylor series expansion, we write 
  \begin{equation}
    F(\mtx{X}) = \trace \sum \limits_{m = 0}^\infty \Expect_{\mtx{A} \sim \mtx{X}}\left[\mtx{A}\right]^m,   
  \end{equation} 
  and thus, in the definitions of \cref{thm:bias_reduction_general}, 
  \begin{equation}
    \alpha_{m} = 1, \quad \beta_{m} \leq \sigma^{-m}, \quad \gamma_m \leq \frac{\sigma^{m}}{1 - \sigma}.       
  \end{equation} 
  Thus, for $k = \left\lfloor -\sqrt{\frac{N}{8}}\log\left(\sigma\right) \right\rfloor$ bootstrap iterations, we attain an upper error bound
  \begin{equation}
    \left|\left(\mathrm{Id} - \rS \right)^{k + 1}F(\mtx{X})\right| \leq \frac{4}{1 - \sigma}\sigma^{\sqrt{N / 32}}.
  \end{equation}
\end{proof}
Our estimates show that the iterated bootstrap achieves a superpolynomial convergence rate.
We point out that an unbiased estimator of the $F$ with bias converging in $\bigO(\sigma^N)$ can be obtained by evaluating the truncated Neumann series in product form and using a different sample for each iteration. 
It remains to be understood whether this discrepancy merely indicates a lack of sharpness in our estimates or arises from the strict ``black-box'' access of the bootstrap to the problem.

\subsection{The nonstationary bootstrap}
\label{sec:nonstationary}
The observation that bootstrap bias correction is an iterative method for inverting $\rS$ poses the question of whether other iterative solvers allow improving its performance.
Indeed, the operator 
\begin{equation}
  [\Bs G](\mtx{X}) \coloneqq (F(\mtx{X}) -  \left(\Expect\limits_{\mtx{Y} \sim_N \mtx{X}} \left[G(\mtx{Y})\right] - G(\mtx{X})\right)
\end{equation}
can be modified by introducing a ``step size'' $\eta \in \R$, resulting in 
\begin{equation}
  [\Bs_{\eta} G](\mtx{X}) \coloneqq (F(\mtx{X}) -  \eta \left(\Expect\limits_{\mtx{Y} \sim_N \mtx{X}} \left[G(\mtx{Y})\right] - G(\mtx{X})\right).
\end{equation}
Choosing a sequence $\{\eta_{i}\}_{i \in \{1, \ldots, k\}}$ we obtain alternative corrections $\Bs_{\eta_{k}}\cdots \Bs_{\eta_{1}} F$, with bias
\begin{equation}
  \Bs_{\eta_{k}} \cdots \Bs_{\eta_{1}} \rS F - F = \left(\prod \limits_{i = 1}^k  \mathrm{Id} - \eta_{i} \rS\right)\left(\rS F - F\right).
\end{equation}
Restriction our attention to moment polynomials, we can again express the bias as 
\begin{equation}
  \Bs_{\eta_{k}} \cdots \Bs_{\eta_{1}} \rS F - F = \sum \limits_{\pi \in \Pi(m)} \mu_{\pi} \left(\left(\prod \limits_{i = 1}^k  \Id - \eta_{i} \mtx{S}\right)\left(\mtx{S}\vct{f}  - \vct{f}\right)\right)_{\pi}.
\end{equation}
For $N \geq m$, appropriate step sizes enable exact bias correction in a finite number of iterations. 
\begin{theorem}
  For $N \geq m$ and 
  \begin{equation}
    \eta_{i} \coloneqq \left(\frac{\frac{N!}{(N - i)!}}{N^i}\right)^{-1} = \left(\prod \limits_{j = N-i + 1}^{N} \frac{j}{N}\right)^{-1},
  \end{equation}
  the nonstationary bootstrap approximation with $k=m$ steps is unbiased. 
\end{theorem}
\begin{proof}
  As illustrated in \cref{fig:bootstrap_moebius}, nonzero entries of $S_{\pi \sigma}$ of $\mtx{S}$ are located either on the diagonal ($\pi = \sigma$) or in the block-lower triangular part ($\# \pi > \# \sigma$), with blocks according to the size of the partitions.
  Furthermore, by \cref{thm:bootstrap_moment_poly}, the diagonal entries are constant in each block.
  The $\{\eta_{i}\}_{1 \leq i \leq m}$ are chosen as reciprocals of these diagonal values.
  Thus, if a vector $\vct{b}$ satisfies that $b_\pi = 0$ for all $\pi \in \Pi(m)$ with $\# \pi \geq k$ then $\left(\left(\Id - \eta_{k} \mtx{S}\right)\vct{b}\right)_{\sigma}$ vanishes for all $\sigma$ with $\# \sigma \geq k + 1$. 
  Successive application of this result and the fact that the $\{\Bs_{\eta_{i}}\}_{1 \leq i \leq m}$ commute imply that 
  \begin{equation}
    \Bs_{\eta_{m}} \cdots \Bs_{\eta_{1}} \rS F - F = \sum \limits_{\pi \in \Pi(m)} \mu_{\pi} \left(\left(\prod \limits_{i = 1}^m  \Id - \eta_{i} \mtx{S}\right)\left(\mtx{S}\vct{f}  - \vct{f}\right)\right)_{\pi} = 0.
  \end{equation}
\end{proof}

\section{Comparison, conclusion, and outlook}

\subsection*{Comparison to U-statistics and related methods}
In \cref{sec:moment_poly_convergence}, we have established the convergence of the iterated bootstrap by showing that as the number of iterations goes to infinity, it converges in expectation to U-statistics of the form \cref{eq:symmetric_statistics}. 
For a moment polynomial $F$ with known coefficients, this formula can be directly evaluated to obtain an unbiased estimator of arbitrary moment polynomials. 
Applied to the special case of cumulant products, this yields the $k$-statistics and polykays.
The resulting estimators attain minimal variance among all unbiased ones \cite{mccullagh2018tensor}.
For non-polynomial examples, they can be applied to polynomial approximate expansions of the objective. 
Using this technique, \cite{zhou2021high} achieves optimal statistical risk even in non-smooth problems. 

In light of the above, one may ask why one should study the bootstrap at all. 
In our view, what sets the bootstrap apart from methods based on symmetric statistics is that it can be applied based on black-box acess to the functional $F$.
While our analysis relies on a specific polynomial approximation, the latter need not be known to apply the algorithm.  
This is reminiscent to the comparison of Krylov subspace methods to Chebychev iteration. 
Although the worst-case convergence rate of the two methods coincides, Krylov methods can utilize subtle problem-specific structure without user intervention \cite{saad2003iterative}. 
Similarly, the bootstrap will always converge according to the best possible choice of polynomial approximation, even though it may be unknown to the user or contain an exponentially large number of terms.

\subsection*{Comparison to Jiao and Han, limitation to smooth functionals}
This work is in many ways complementary to that of \cite{jiao2020bias}.
They focus on a specific probability distribution (the binomial model) but arbitrary, even non-smooth functionals. 
Instead, we focus on arbitrary distributions but restrict ourselves to moment polynomials.
The finite support of the binomial model allows \cite{jiao2020bias} to represent any functional exactly as a polynomial on the support of the probability distribution. 
Instead, our work uses smoothness properties of the functional to obtain polynomial approximations.
They study the limit of either a constant number of data points and infinite data, or constant data and infinite bootstrap iterations. 
The former achieves an algebraic bias convergence determined by the number of bootstrap iterations and the smoothness of the functional, while the latter achieves the Legendre polynomial approximation.
Instead, our core result is that scaling the number of iterations with the square root of the available data allows (for sufficiently smooth functionals) to achieve a superalgebraic bias convergence rate.
In principle any continuous functional can be approximated by polynomials. 
But to be useful, our estimates require that the polynomial approximations converge sufficiently fast, ruling out non-smooth functionals. 
This is because while every bootstrap iteration decreases the bias in the leading $\approx \sqrt{N}$ terms of the polynomial approximation by a constant factor, it generally increases the bias in the remaining terms by a constant factor. 
Thus, for slowly decaying polynomial approximation errors, large numbers of bootstrap iteration do not help. 
This reconciles our work with the negative results of \cite[Corollary 2]{jiao2020bias}.

\section*{Conclusion and outlook} This work reveals the relationship between bootstrap resampling methods and M{\"o}bius inversion.
This insight allows us to derive nonasymptotic convergence results for the iterated bootstrap and motivates new variants of the bootstrap that more aggressively target the bias of low-order moment polynomials.
At present, the results in this work are mostly of theoretical interest, as the additional variance incurred by naive iterated bootstrapping (and especially the variant introduced in \cref{sec:nonstationary}) is often prohibitive.
Exploring the algorithmic implications of the insights presented in this work is an exciting direction for future research.
Jackknife bias correction is closely related to the bootstrap, making it a promising future target for the combinatorial approach presented in this work. 
From a combinatorial perspective, the characteristic feature of free probability is that it replaces lattice of set partitions with the lattice of noncrossing partitions \cite{nica2006lectures}.
This poses the question of whether non-crossing variants of $\mtx{S}$ define a meaningful notion of ``sampling'' in this context.

\section*{Acknowledgments}
FS gratefully acknowledges support from the Office of Naval Research under award number N00014-23-1-2545 (Untangling Computation).
This work was produced using Microsoft Copilot for autocompletion and spelling/grammar checking.
We thank the anonymous reviewers for their comments and suggestions that helped us improve our manuscript significantly.
We also thank Hongjian Lan and Yucong Liu for helpful feedback.

\bibliographystyle{siamplain}
\bibliography{references}

\end{document}